\documentclass[12pt,reqno]{amsproc}
\usepackage[english]{babel}
\usepackage{amssymb}
\usepackage{amsmath}
\usepackage{amsthm}
\usepackage{graphicx}
\usepackage{bbm}
\usepackage{mathtools}
\usepackage{tikz}
\usepackage{tikz-cd}
\usepackage{geometry}
\usepackage{xcolor}
\usepackage{enumitem}
\usepackage{microtype}
\usepackage{mathdots}	% provides \iddots
\usepackage{comment}
\usepackage{xifthen}
\allowdisplaybreaks
\usetikzlibrary{shapes.geometric,fit}

\definecolor{myurlcolor}{rgb}{0,0,0.4}
\definecolor{mycitecolor}{rgb}{0,0.5,0}
\definecolor{myrefcolor}{rgb}{0.5,0,0}
\usepackage[pagebackref]{hyperref}
\hypersetup{colorlinks,
linkcolor=myrefcolor,
citecolor=mycitecolor,
urlcolor=myurlcolor}

\usepackage[noabbrev,capitalise]{cleveref}

\numberwithin{equation}{section}

\theoremstyle{plain}
\newtheorem{dummy}{}[subsection]
\newtheorem{thm}[dummy]{Theorem}
\newtheorem{lemma}[dummy]{Lemma}
\newtheorem{prop}[dummy]{Proposition}
\newtheorem{cor}[dummy]{Corollary}

\newtheorem{defn}[dummy]{Definition}

\theoremstyle{definition}
\newtheorem{remark}[dummy]{Remark}
\newtheorem{ex}[dummy]{Example}

\Crefname{ex}{Example}{Examples}	% to get the plural to appear
\Crefname{thm}{Theorem}{Theorems}
\Crefname{prop}{Proposition}{Propositions}
\Crefname{cor}{Corollary}{Corollaries}
\Crefname{defn}{Definition}{Definitions}

\newcommand{\N}{\mathbb{N}}

\newcommand{\R}{\mathbb{R}}

\newcommand{\cat}[1]{{\mathsf{#1}}} % font for categories
\newcommand{\ar}[2][]{\arrow{#2}{#1}}
\newcommand{\uni}[2][]{\arrow[dashrightarrow]{#2}{#1}} % in diagrams - use like \arrow
\newcommand{\id}[1][]{\ifthenelse{\equal{#1}{}}{\mathrm{id}}{\mathrm{id}_{#1}}} % identity, with object as subscript if optional argument is given
\newcommand{\op}{\mathrm{op}}

\DeclareMathOperator{\Barc}{Bar}	% bar construction
\newcommand{\barc}[2]{\Barc_{#1}(#2)}	% bar construction \barc{monad}{algebra}
\newcommand{\scdots}[2][]{\mathinner{#1\overset{#2}{\cdots}#1}}

\newcommand{\Set}{\cat{Set}}

\let\originalleft\left
\let\originalright\right
\renewcommand{\left}{\mathopen{}\mathclose\bgroup\originalleft}
\renewcommand{\right}{\aftergroup\egroup\originalright}

\tikzset{% 
    bullet/.style={
       fill=black,
       circle,
       minimum width=1pt,
       inner sep=1pt
     },
     relation/.style={
       -,
       thick,
       shorten <=2pt,
       shorten >=2pt
     },
     function/.style={
       ->,
       thick,
       shorten <=2pt,
       shorten >=2pt
     },
     every fit/.style={
       ellipse,
       draw,
       inner sep=0pt
     }
}

\usepackage{enumitem}
\setlist[enumerate]{label=(\alph*),itemsep=5pt,topsep=8pt}
\setlist[itemize]{label=$\triangleright$,itemsep=5pt,topsep=6pt}

\Crefformat{enumi}{#2#1#3}
\Crefname{equation}{}{}		% abbreviate 'Equation (3.2)' to '(3.2)'

\author[C.~Constantin]{Carmen Constantin}
\address{Mansfield College, Oxford, UK}
\email{carmen.constantin@mansfield.ox.ac.uk}

\author[T.~Fritz]{Tobias Fritz}
\address{Department of Mathematics, University of Innsbruck, Austria}
\email{tobias.fritz@uibk.ac.at}

\author[P.~Perrone]{Paolo Perrone}
\address{Department of Computer Science, University of Oxford, United Kingdom}
\email{paolo.perrone@cs.ox.ac.uk}

\author[B.~Shapiro]{Brandon T. Shapiro}
\address{Topos Institute, Berkeley CA, U.S.A}
\email{shapiro@topos.institute}

\usepackage{tikz}

\title[Partial evaluations and the bar construction]{Partial evaluations and the compositional structure of the bar construction \vspace{3pt}}

\begin{document}

\begin{abstract}
	\vspace{3pt}

	The algebraic expression $3 + 2 + 6$ can be evaluated to $11$, but it can also be \emph{partially evaluated} to $5 + 6$. In categorical algebra, such partial evaluations can be defined in terms of the $1$-skeleton of the bar construction for algebras of a monad. We show that this partial evaluation relation can be seen as the relation internal to the category of algebras generated by relating a formal expression to its total evaluation. The relation is transitive for many monads which describe commonly encountered algebraic structures, and more generally for BC monads on $\Set$ (which are those monads for which the underlying functor and the multiplication are weakly cartesian). We find that this is not true for all monads: we describe a finitary monad on $\Set$ for which the partial evaluation relation on the terminal algebra is not transitive. 
	
	With the perspective of higher algebraic rewriting in mind, we then investigate the compositional structure of the bar construction in all dimensions. We show that for algebras of BC monads, the bar construction has fillers for all \emph{directed acyclic configurations} in $\Delta^n$, but generally not all inner horns. 	%We introduce several additional \emph{completeness} and \emph{exactness} conditions on simplicial sets which correspond via the bar construction to composition and invertibility properties of partial evaluations, including those arising from \emph{weakly cartesian} monads. 
%
%	In order to provide simplified presentations of these conditions and relate them to more familiar properties of simplicial sets, we characterize and produce factorizations of pushouts and certain commutative squares in the simplex category. Most notably, the properties we consider correspond to weak analogues of the exactness properties characterizing the $1$- and $2$-Segal conditions of Dyckerhoff and Kapranov, as well as additional exactness conditions studied by G\'{a}lvez, Kock, and Tonks.
\end{abstract}

%%% title page, newgeometry makes the title be placed higher up so that the abstract and table of contents fit onto the page
\newgeometry{top=0.2cm,bottom=1cm}
\maketitle
\thispagestyle{empty}
\setcounter{tocdepth}{1}
\tableofcontents
\vspace{-1cm}

\restoregeometry

\newpage

\section{Introduction}

In this paper we study compositional and combinatorial aspects of the
bar construction for algebras of several types of monads,
motivated by the idea that edges in the bar construction can be interpreted as \emph{partial evaluations} of formal algebraic expressions \cite{FP}. %This leads us to introduce various new classes of simplicial sets with a compositional flavour similar to but distinct from quasicategories. A careful analysis of pushouts in the simplex category $\Delta$ lets us characterize these classes of simplicial sets in terms of filler conditions.
This partly involves the classes of simplicial sets introduced by us in the companion paper~\cite{delta_squares}.

\subsection*{Partial evaluations and the bar construction}

In more detail, the \emph{bar construction} associates to every Eilenberg--Moore algebra a simplicial object in the category of algebras, playing the role of a universal resolution of the algebra~\cite{trimble}. For the case of monads on $\Set$, the resulting simplicial set can be interpreted operationally in terms of \emph{partial evaluations}. Just as one can say that the formal sum $2+3+4$ can be evaluated to $9$, one can also say that it can be \emph{partially evaluated} to $5+4$. A way to make this precise is by using the 1-dimensional structure of the bar construction of $\N$ as an algebra of a suitable monad, namely the monad of commutative monoids (see \Cref{trees_and_boxes}), in which such a partial evaluation is represented by an edge (1-simplex).
More generally, the bar construction can be seen as a simplicial set where the 0-simplices are formal expressions specified by the monad, the 1-simplices are partial evaluations between two such formal expressions, and the higher-dimensional simplices have to do with higher substitutions; for example, we will interpret the 2-simplices as composition rules for partial evaluations. 
%Indeed, if $(T,\mu,\eta)$ is the monad of commutative monoids, and $(\N,e:T\N\to\N)$ is the $T$-algebra defined by natural numbers and addition, we can see the set $T\N$ as the one containing formal sums of elements of $\N$. Two formal sums $p$ and $q$ in $T\N$ are then related by a partial evaluation if there exists a doubly formal sum (i.e.~an element $a$ of $TT\N$) such that $\mu(a)=p$ and $Te(a)=q$. For our sum example $a$ is given by the formal sum of $2+3$ and $4$ (as elements of $T\N$, so that their formal sum lies itself in $TT\N$).
The relation induced by the existence of partial evaluations can be seen as the relation internal to the category of algebras generated by linking a formal expression to its result. Partial evaluations themselves can be seen them as a categorification (or a \emph{proof-relevant version}) of this idea.

We will present these ideas in more detail in \Cref{secpev,trees_and_boxes}. We also refer to our earlier work \cite{FP}, and note that an example of a partial evaluation in the context of the bar construction had already appeared earlier in notes by Baez on cohomology and computation~\cite{baez_notes}.
% A more theoretical treatment of the bar construction can be found in \cite{trimble}.

\subsection*{Compositional structure of partial evaluations}

Whenever the monad under consideration is cartesian, which happens for example whenever it is presented by a non-symmetric operad, the bar construction of every algebra is known to be the nerve of a category (\Cref{cartesian_nerve}).
This means, in particular, that partial evaluations can be composed uniquely, and that their composition operation is strictly unital and associative. A large class of monads appearing in algebra, as well as in probability and other fields, are however only \emph{weakly cartesian}, or have even less rigid properties such as the property BC (see \Cref{secmonads}). Examples of weakly cartesian monads on $\Set$ are all those monads which are presented by a symmetric operad, such as the monad of commutative monoids. 

For the algebras of BC monads, the bar construction generally satisfies weaker filling conditions than the nerve of a category, making the resulting composition operation no longer well defined. As we will see, the bar construction of these monads is generally not even a quasicategory, as not all inner horns admit a filler above dimension 2. It nevertheless satisfies filler conditions reminiscent of a compositional structure;
in the companion paper~\cite{delta_squares}, we have introduced \emph{inner span complete} simplicial sets with the current application in mind: we prove that bar constructions of BC monads are inner span complete simplicial sets. 
This has powerful consequences for their compositional structures, since \cite[Theorem~5.14]{delta_squares} implies the existence of a large class of fillers for inner span complete simplicial set.

For example, the convex-combination monad---also known as the distribution monad in probability terms---is BC, and therefore its algebras have inner span complete bar constructions. Building on the relation to second-order stochastic dominance developed in~\cite{FP}, it is natural to wonder whether this inner span completeness also has significance for probability theory,
a question which we do not yet have an answer for.

\subsection*{Outline}

\begin{itemize}
	\item In \Cref{secmonads} we provide some background on monads and some standard definitions, such as that of a weak pullback. We also outline the relevant weak exactness conditions on monads (\Cref{liftmonads}) and their algebras (\Cref{liftalgebras}) that we use in the rest of the paper. Some of these conditions are standard, while some appear for the first time here (to our knowledge).
 
	\item In \Cref{secpev} we recall the concept of partial evaluations from~\cite{FP}, which can be seen as an operational interpretation of the bar construction in low dimension, and study its compositional properties further. In particular, we define the partial evaluation relation and show that it is the smallest relation internal to the category of algebras which relates a formal expression to its total evaluation (\Cref{pe_internal}), and then we proceed to give some criteria for when and how partial evaluations can be composed (\Cref{composition}) and for when they can be reversed  (\Cref{reversing,irreversibility}). 
 
	\item In \Cref{secbar} we formally introduce the bar construction (\Cref{barconstruction}) and begin our study of its compositional properties. We use the commutative-monoid monad to give counterexamples to several natural hypotheses, including the general
		non-uniqueness of composites (\Cref{nonuniqueness}) and the 
		nonexistence of fillers for inner horns (\Cref{no_horn_fillers}). This prepares the ground for the last part of the paper, in which we give a number of compositional properties which \emph{do} hold for the bar constructions of various classes of monads which include the commutative-monoid monad.

	\item In \Cref{seccompositional}, we recall the \emph{inner span complete} simplicial sets from~\cite{delta_squares} and discuss some of their basic properties. We then show that the bar construction of an algebra is:
		\begin{itemize}
			\item inner span complete for a BC monad (\Cref{BC_to_ISpC});
			\item inner span complete and stiff for a weakly cartesian monad (\Cref{wc2ic});
			\item inner span complete and split for a weakly cartesian and strictly positive monad (\Cref{pwc2pc}).
		\end{itemize}

\end{itemize}

%Here is how the sections depend on one another.
%$$
%\begin{tikzcd}
%	& \textrm{\Cref{secmonads}} \ar{d}  \\ 
%	& \textrm{\Cref{secpev}} \ar{d} \\
%	& \textrm{\Cref{secbar}} \ar{ddl} \ar{dd} \ar{ddr} \ar{ddrr} & \textrm{\Cref{sec:pushouts_delta}} \ar{ddll} \ar{ddl} \ar{dd} \ar{ddr} \\ \\
%	 \textrm{\Cref{sec:bc_bar}} \ar[dashed]{r} & \textrm{\Cref{sec:bcwc}} \ar[dashed]{r} & \textrm{\Cref{sec:bcp}} \ar[dashed]{r} & \textrm{\Cref{sec:bcinv}}
%\end{tikzcd}
%$$
%Here, a normal arrow denotes an actual mathematical dependency, while a dashed arrows denotes dependency on a conceptual and intuitive level which will merely help with understanding but is not technically necessary.

\subsection*{Relevant background} We assume familiarity with the theory of monads, their algebras, and the basic idea of how to do categorical algebra in terms of finitary monads on $\Set$. We also assume familiarity with simplicial sets, but provide a brief recap next in the context of setting up notation. Some parts also assume familiarity with the basic definitions of quasicategory theory~\cite{joyal}.

\subsection*{Notation and terminology}

Throughout the paper, $\Delta$ denotes the simplex category, i.e.~the category of nonempty finite ordinals
\[
	{[n]} \coloneqq \{0,\ldots,n\}	
\]
for $n \in \N$ as objects and monotone maps as morphisms. Similarly, $\Delta_+$ denotes the augmented simplex category, i.e.~the category of finite ordinals and monotone maps, where we also include the empty ordinal ${[-1]} \coloneqq \emptyset$. In either case, its \emph{generating coface maps} are the morphisms
\[
	d^{n,i} \: : \: {[n-1]} \longrightarrow {[n]}
\]
for $i = 0,\ldots,n$, given by the inclusion of ${[n-1]}$ into ${[n]}$ omitting the element $i$. The \emph{generating codegeneracy maps} are likewise the morphisms
\[
	s^{n,i} \: : \: {[n+1]} \longrightarrow {[n]}
\]
for $i = 0,\ldots,n$, given by the map which hits $i$ twice but otherwise acts like the identity. A coface map or codegeneracy map in general is a composite of generating ones. 

A simplicial set is then a functor $\Delta^\op \to \Set$, and an augmented simplicial set is a functor $\Delta_+^\op \to \Set$. As usual, when the simplicial set under consideration is clear from the context, then we denote the face and degeneracy maps (the functor's action on coface and codegeneracy maps) using subscripts, $d_{n,i}$ and $s_{n,i}$, or merely $d_i$ and $s_i$.

We generally specify a finitary monad on $\Set$ in terms of the algebraic theory that it presents. For example, the \emph{commutative-monoid monad} will be used throughout the paper for illustration.

\subsection*{Acknowledgements}

We first of all thank Joachim Kock and an anonymous referee for detailed comments on an earlier version, which have resulted in various improvements to the exposition.

This paper originates from the \emph{Applied Category Theory 2019} school. We thank the organizers Daniel Cicala and Jules Hedges for having made it happen, the Computer Science Department of the University of Oxford for hosting the event, as well as all other participants of the school for the interesting discussions and insights, especially Martin Lundfall. 

Research for the third author was partly funded by the Fields Institute (Canada), and by the AFOSR grants FA9550-19-1-0113 and FA9550-17-1-0058 (U.S.A.). The fourth author is supported by the Department of Defense (DoD) through the National Defense Science and Engineering Graduate (NDSEG) Fellowship Program.

\section{Preliminaries on monads}\label{secmonads}

In this section we list some conditions on monads and on their algebras, which in the rest of this work will be both applied and given an operational interpretation in terms of partial evaluations.
Some of these conditions are known in the literature (see the given references), and some are introduced here for the first time.

\subsection{(Weakly) cartesian squares}\label{weakly_cartesian}

We start with some preliminary considerations on weak pullbacks.

\begin{defn}\label{weak pullback}
 (\cite{JoyalAnalytic}) A diagram 
 \begin{equation}\label{square}
 \begin{tikzcd}
	A \ar{r}{f} \ar{d}[swap]{g}  & B \ar{d}[swap]{m} \\
	C \ar{r}{n} & D
 \end{tikzcd}
 \end{equation}
 in a category $\cat{C}$ is called a \emph{weak pullback}, or \emph{weakly cartesian square}, if for every object $S$ and every commutative diagram
 $$
 \begin{tikzcd}[sep=small]
	S \ar{drrr}{p} \ar{dddr}[swap]{q} \\
	& && B \ar{dd}[swap]{m} \\ \\
	& C \ar{rr}{n} && D
 \end{tikzcd}
 $$
 in $\cat{C}$ there exists an arrow $S\to A$ making the following diagram commute. 
 $$
 \begin{tikzcd}[sep=small]
	S \ar{drrr}{p} \ar{dr} \ar{dddr}[swap]{q} \\
	& A \ar{rr}[swap]{f} \ar{dd}{g}  && B \ar{dd}[swap]{m} \\ \\
	& C \ar{rr}{n} && D
 \end{tikzcd}
 $$
\end{defn}

If we are in the category $\cat{Set}$, then the diagram~\Cref{square} is a weak pullback if and only if for every $b\in B$ and $c\in C$ with $m(b)=n(c)$ there exists $a\in A$ such that $f(a)=b$ and $g(a)=c$.  Note that if we moreover require the map $S\to A$ to be unique, then we get the ordinary notion of pullback (or cartesian square). We sometimes also say \emph{strong} pullback to emphasize the distinction with weak pullbacks.

Strong pullbacks satisfy the following standard \emph{pullback lemma}, also known as the \emph{prism lemma} in its homotopical version (see for instance \cite[Lemma 1.11]{GKT1} in the homotopical version).

\begin{lemma}\label{pullback_lemma}
In any diagram as below, if the right square and outer rectangle are strong pullbacks, then so is the left square.
$$
\begin{tikzcd}
	\cdot \ar{r} \ar{d} & \cdot \ar{r} \ar{d} & \cdot \ar{d} \\
	\cdot \ar{r} & \cdot \ar{r} & \cdot
\end{tikzcd}
$$
\end{lemma}

A fundamental difference between strong and weak pullbacks is that this does not hold for weak pullbacks.

\begin{ex}
Consider the diagram below in $\Set$:
$$
\begin{tikzcd}
	\{*\} \ar{drr}{b} \ar[equals]{ddr} \\
	& \{*\} \ar{r}[near start]{a} \ar[equals]{d} & \{a,b\} \ar{r} \ar{d} & \{*\} \ar[equals]{d} \\
	& \{*\} \ar[equals]{r} & \{*\} \ar[equals]{r} & \{*\}
\end{tikzcd}
$$

Both the right square and the outer rectangle are weak pullbacks, and the kite shaped subdiagram commutes, but there is no map $h : \{*\} \to \{*\}$ with $ah=b$. The left square is therefore not a weak pullback.
\end{ex}

The following lemma will be particularly useful when $f$ or $g$ is a degeneracy map of a simplicial set, which is always (split) monic.

\begin{lemma}\label{mono_strong_pullback}
If the square below is a weak pullback in any category and $f$ or $g$ is monic, then the square is a strong pullback.
$$
 \begin{tikzcd}
	A \ar{r}{f} \ar{d}[swap]{g}  & B \ar{d}[swap]{m} \\
	C \ar{r}{n} & D
 \end{tikzcd}
$$
\end{lemma}

\begin{proof}
Assume $f$ is monic (the argument for $g$ is analogous), and let $p : S \to B$, $q : S \to C$ be maps that commute over $D$. Any two induced maps $h,h' : S \to A$ with $fh=fh'=p$ are equal as $f$ is monic.
\end{proof}

\begin{defn}[{e.g.~\cite[Appendix, Definition 4]{JoyalAnalytic}}] Let $F:\cat{C}\to\cat{D}$ be a functor. We call $F$ \emph{cartesian} if it preserves pullbacks, and \emph{weakly cartesian} if it preserves weak pullbacks.
\end{defn}

Examples of weakly cartesian functors include Joyal's analytic functors (\cite[Appendix, Theorem 1]{JoyalAnalytic}). If $\cat{C}$ has pullbacks, then $F:\cat{C}\to\cat{D}$ is weakly cartesian equivalently if it sends pullbacks to weak pullbacks.

\begin{defn}[{e.g.~\cite[Definition 2]{JoyalAnalytic}}] Let $F,G:\cat{C}\to\cat{D}$ be functors. A natural transformation $\alpha:F\Rightarrow G$ is called \emph{cartesian} (resp.~\emph{weakly cartesian}) if for every morphism $f:X\to Y$ of $\cat{C}$, the naturality square
 $$
 \begin{tikzcd}
	FX \ar{r}{Ff} \ar{d}[swap]{\alpha_X} & FY \ar{d}[swap]{\alpha_Y} \\
	GX \ar{r}{Gf} & GY
 \end{tikzcd}
 $$
 is cartesian (resp.~weakly cartesian).
\end{defn}

\subsection{Lifting conditions for monads}\label{liftmonads}

We now define the various properties of monads relating to (weakly) cartesian squares which we use throughout the paper. The interested reader can find more details in \cite{weakweber} and \cite{CHJ}.

\begin{defn}
A monad $(T,\eta,\mu)$ is called \emph{BC}  if $T$ preserves weak pullbacks and the multiplication $\mu$ is weakly cartesian.
\end{defn}

``BC'' stands for ``Beck--Chevalley'', and follows the terminology of \cite{CHJ}. As we will see, the BC property, and in particular weak cartesianness of $\mu$, is closely related to the problem of composing partial evaluations (see \Cref{composition} for the details).

\begin{ex}\label{distBC}
The convex-combination monad or distribution monad,\footnote{See e.g.~\cite[Section~6.2]{FP} for the detailed definition.} usually denoted by $D$, is BC. It is known that the multiplication transformation is weakly cartesian (\cite[Proposition~6.4]{FP}). 
To show moreover that the functor $D$ preserves weak pullbacks, we will use a construction sometimes known as \emph{conditional product}.\footnote{We refer to Simpson's~\cite[Section~6]{simpson_independence} for a categorical treatment which is especially close to what we use here.} Let 
$$
 \begin{tikzcd}
	A \ar{r}{f} \ar{d}[swap]{g}  & B \ar{d}[swap]{m} \\
	C \ar{r}{n} & E
 \end{tikzcd}
$$
be a (strong) pullback in $\cat{Set}$. In particular, we have 
\begin{equation}\label{fiberprod}
A \cong \coprod_{e\in E} m^{-1}(e) \times n^{-1}(e) .
\end{equation}
Consider its $D$-image,
$$
 \begin{tikzcd}
	DA \ar{r}{Df} \ar{d}[swap]{Dg}  & DB \ar{d}[swap]{Dm} \\
	DC \ar{r}{Dn} & DE .
 \end{tikzcd}
$$
Now let $p\in DB$ and $q\in DC$ be finitely supported distributions, and suppose that $Dm(p)=Dn(q)$, i.e.~that for all $e\in E$,
$$
\sum_{b\in m^{-1}(e)} p(b) = \sum_{c\in n^{-1}(e)} q(c) .
$$
Denote by $r\in DE$ the resulting distribution on $E$. Now define the distribution $s\in DA$ as follows. Using \Cref{fiberprod}, we can write every element of $A$ as a pair $(b,c)$, with $b\in B$ and $c\in C$ such that $m(b)=n(c)$. Now, for each such $(b,c)$, let $e\coloneqq m(b)=n(c)$, and set
$$
s(b,c) \coloneqq \begin{cases}
		\dfrac{p(b) \cdot q(c)}{r(e)}	& \textrm{if } r(e) > 0,	\\
		0				& \textrm{otherwise}.
		\end{cases}
$$
It is straightforward to verify that this satisfies the relevant normalization condition $\sum_{(b,c) \in A} s(b,c) = 1$ to qualify as a probability distribution. We then have that for each $b\in B$,
\begin{align*}
Df(s)(b) &= \sum_{(b,c)\in f^{-1}(b)} s(b,c) = \sum_{c\in C  \mbox{\tiny{ s.t.~}} m(b)=n(c)} \dfrac{p(b) \cdot q(c)}{r(e)} \\
& = p(b) \, \dfrac{\sum_{c\in n^{-1}(e)} q(c)}{\sum_{c\in n^{-1}(e)} q(c)} = p(b) ,
\end{align*}
and analogously $Dg(s)(c)=q(c)$. Hence $D$ is indeed a weakly cartesian functor.
\end{ex}

\begin{defn}
	A monad $(T,\eta,\mu)$ is called \emph{cartesian} if $T$ preserves pullbacks and both $\eta$ and $\mu$ are cartesian.
	It is called \emph{weakly cartesian} if $T$ preserves weak pullbacks and both $\eta$ and $\mu$ are weakly cartesian.
\end{defn}

In other words, a weakly cartesian monad is a BC monad for which also the unit $\eta$ is weakly cartesian. Taking into account that the components of $\eta$ are typically monomorphisms, \Cref{mono_strong_pullback} then shows that the naturality squares of $\eta$ are strongly cartesian whenever this holds.\footnote{For example, the terminal monad $T$, for which $TX$ is singleton for every set $X$, does not have monomorphisms components for $\eta$.}

\begin{ex}
	\label{operads_etc}
	The monad of monoids is cartesian, and more generally, every monad arising from a non-symmetric operad is cartesian. This was first understood by Weber, see \cite[Proposition~2.6 and Example~2.7.1]{weakweber}, where these ideas are presented in a somewhat different language. Explanations in terms of more similar concepts to the ones presented in this work are	\cite[Observation~2.1(d)]{CHJ} for the monoid monad case, and \cite[Section~C.1]{leinster} for the general statement. 
\end{ex}

\begin{ex}
    \label{sym_operads_etc}
	Similarly, the monad of commutative monoids is weakly cartesian, and more generally, every monad arising from a symmetric operad is weakly cartesian. 
	In general, these monads are not cartesian. To illustrate why, let's consider the example of the commutative-monoid monad. 
    Consider a noninjective function between sets $f:X\to Y$, and form the the naturality square with $\mu$. 
    \[
     \begin{tikzcd}
      TTX \ar{r}{TTf} \ar{d}{\mu} & TTY \ar{d}{\mu} \\
      TX \ar{r}{Tf} & TY
     \end{tikzcd}
    \]
    Now let $x_0,x_1,x_2\in X$ be such that $f(x_0) \neq f(x_1)=f(x_2)$. Denote $y_0 \coloneqq f(x_0)$ and $y \coloneqq f(x_1)=f(x_2)$.
    Consider now the elements\footnote{The boxes denote levels of formality corresponding to applications of $T$, a notation that will be explained in detail in \Cref{trees_and_boxes}.}
    \[
	    t \coloneqq \boxed{x_0}+\boxed{x_1}+\boxed{x_2}\in TX, \qquad \sigma \coloneqq \boxed{\boxed{y_0}+\boxed{y}}+\boxed{\boxed{y}}\in TTY.
    \]
    As one can directly compute, these terms are mapped to
    \[
	    (Tf)(t)=\mu(\sigma)=\boxed{y_0}+\boxed{y}+\boxed{y} \in TY
    \]
    in the lower right corner of the diagram.
    Now, since we can rearrange the terms \emph{within} the boxes, but not \emph{between} the boxes, the following elements of $TTX$ are distinct.
    \[
     \alpha \coloneqq \boxed{\boxed{x_0}+\boxed{x_1}}+\boxed{\boxed{x_2}} , \qquad 
     \beta \coloneqq \boxed{\boxed{x_0}+\boxed{x_2}}+\boxed{\boxed{x_1}}.
    \]
    However, for both elements, we have that 
    \[
     \mu(\alpha) = t = \mu(\beta) , \qquad TTf(\alpha) = \sigma = TTf(\beta) .
    \]
    Therefore the naturality diagram above can only be a weak pullback, not a pullback. 

    Again, this phenomenon was first understood by Weber \cite[Section~11 and Example~2.7.5]{weakweber}, and a treatment more similar in language to this work is given in \cite[Example~8.2]{CHJ} and~\cite{analyticmonads}. Further examples and nonexamples of weakly cartesian monads can be found again in \cite{CHJ}.
\end{ex}

\begin{ex}
	\label{distribution_monad_not_wc}
	Although it is BC, the distribution monad $D$ is not weakly cartesian since its unit $\eta$ is not weakly cartesian.
	Indeed consider the naturality square for the unique map $X \to \{\ast\}$ with any set $X$.
	\[
		\begin{tikzcd}
			X \ar{d}{\eta} \ar{r}	& \{\ast\} \ar{d}{\eta}	\\
			DX \ar{r}		& D\{\ast\}
		\end{tikzcd}
	\]
	Since $D\{\ast\}$ is again a one-element set, the right arrow is an isomorphism.
\end{ex}

While the definitions above have previously appeared in the literature, the following are new (as far as we know).

\begin{defn}
	\label{strictly_pos}
 A monad $T$ is called \emph{strictly positive} if the square 
 \begin{equation}\label{possquare}
 \begin{tikzcd}
	X \ar[equals]{d}{} \ar{r}{\eta\eta} & TTX \ar{d}[swap]{\mu} \\
	X \ar{r}{\eta} & TX. 
 \end{tikzcd}
 \end{equation}
 is a pullback for all $X$.
\end{defn}

Since the left vertical map is an identity, the diagram above is a pullback if and only if it is a weak pullback.

For a cartesian monad, we next show that it is enough to check this condition on the terminal set $X = 1 = \{\ast\}$, so that it is strictly positive if and only if $\eta\eta(\ast)$ is the only element of $TT1$ which multiplies to $\eta(1)$.

\begin{prop}
 \label{cartesian_positive}
 Let $(T,\eta,\mu)$ be a monad on $\cat{Set}$ such that $\eta$ or $\mu$ is cartesian. 
 Then the square \Cref{possquare} is a pullback for all sets $X$ if and only if it is for $X=1$.
\end{prop}
\begin{proof}
	The ``only if'' direction is trivial. For the ``if'' direction, suppose that the square \Cref{possquare} is a pullback for $X=1$. Let $X$ be any set, and denote by $u:X\to 1$ the unique map. We can enlarge \Cref{possquare} to the following diagram.
\begin{equation}
 \label{double_eta}
 \begin{tikzcd}
	X \ar{dr}{u} \ar[equals]{ddd} \ar{rrr}{\eta\eta} &&& TTX \ar{dl}[swap]{TTu} \ar{ddd}[swap]{\mu} \\
	& 1 \ar{r}{\eta\eta} \ar[equals]{d} & TT1 \ar{d}[swap]{\mu} \\
	& 1 \ar{r}{\eta} & T1 \\
	X \ar{rrr}{\eta} \ar{ur}{u} &&& TX \ar{ul}[swap]{Tu}
 \end{tikzcd}
\end{equation}
If $\mu$ is cartesian, then both the left and right squares in \Cref{double_eta} are pullbacks, hence so is the outer square by the pullback lemma (\Cref{pullback_lemma}).  Likewise, if $\eta$ is cartesian, then both the top and bottom squares in \Cref{double_eta} are pullbacks, hence so is the outer square by the pullback lemma.
\qedhere
\end{proof}

As we will see (\Cref{irreversibility}), the strict positivity condition has significance for partial evaluations by giving conditions for when these are ``irreversible''.
But the following examples are what primarily motivates the terminology.

\begin{ex}
	\label{monoid_spos}
	For $M$ a monoid, consider the $M$-set monad $M \times -$ on $\Set$. Per the above, this monad is strictly positive if and only if the diagram
	\[
		\begin{tikzcd}
			1 \ar[equals]{d} \ar{r}		& M \times M \ar{d}	\\
			1 \ar{r}		& M
		\end{tikzcd}
	\]
	where the arrows denote the obvious structure maps, is a pullback. In other words, the monad $M \times -$ is strictly positive if and only if the unit element of $M$ cannot be factored nontrivially. In other words, if in additive notation for $m,n\in M$ we have $m+n=0$, then $m=n=0$. Intuitively, there are no negative elements.
\end{ex}

\begin{ex}
	\label{nonpositive}
	The monads of monoids and commutative monoids are not strictly positive. Indeed with $X$ a set and any $x \in X$, consider $x \in X$ itself together with the ``doubly formal expression''
	\[
		\boxed{\boxed{x}} + \boxed{\bullet} \;\in\; TTX,
	\]
	where $\bullet \in TX$ denotes the neutral element, and each box denotes a level of formality corresponding to an application of $T$. (We will develop this notation for elements of $T^n X$ more formally in \Cref{trees_and_boxes}.) Then these two elements show that the square \Cref{possquare} is not a pullback, since both elements map to $\boxed{x} \in TX$, but the doubly formal element under consideration differs from $\eta\eta(x) = \boxed{\boxed{x}}$.
\end{ex}
 
\begin{ex}
	\label{semigroups}
	On the other hand, the semigroup monad $T$ is strictly positive, as a consequence of the previous \Cref{cartesian_positive}: it is the monad associated to a non-symmetric operad and therefore cartesian; furthermore, we have $T1 = \N_{> 0}$, and $TT1$ can therefore be identified with the set of nonempty lists of positive integers, in such a way that $\mu : TT1 \to T1$ is the map which takes a list of positive integers and forms their sum. Based on this, it is straightforward to see that the strict positivity condition holds for the object $1$.
	
	Similarly, the commutative-semigroup monad is both weakly cartesian and strictly positive, but not strongly cartesian. 
	Indeed it is weakly cartesian by virtue of being the monad associated to a symmetric operad (\Cref{sym_operads_etc}), namely the one with exactly one operation in each positive arity and no operation in arity zero.
	Moreover, the unit of the monad is monic, which implies that $\eta$ is a cartesian transformation. Hence by \Cref{cartesian_positive}, we only need to check that the square \Cref{possquare} is cartesian for $X=1$, which holds in an analogous manner to the semigroup monad case with multisets instead of lists.
\end{ex}

\begin{ex}
    \label{dist_monad_pos}
 The distribution monad of \Cref{distBC} is strictly positive. Indeed, let $X$ be a set, let $x\in X$ and let $\pi\in DDX$ be such that $\mu(\pi)=\eta(x)$, meaning that for all $y\in X$,
 $$
 \mu(\pi)(y) = \sum_{p\in DX} \pi(p)\,p(y) = \eta(x)(y) = \begin{cases}
                                                            1 & x=y , \\
                                                            0 & x\ne y .
                                                          \end{cases}
 $$
 Now the function $\eta(x):X\to[0,1]$, which returns $1$ at $x$ and $0$ elsewhere, cannot be expressed as a nontrivial convex-combination of other (positive normalized) functions (in convex analysis terminology, it is ``extremal''). Therefore the only possibility is that for all $p\in DX$, 
 $$
 \pi(p) = \eta(\eta(x)) = \begin{cases}
                            1 & p=\eta(x) , \\
                            0 & p\ne \eta(x) .
                          \end{cases}
 $$
 This makes the diagram \eqref{possquare} a pullback for the distribution monad $D$. 
\end{ex}

\subsection{Lifting conditions for algebras}\label{liftalgebras}

So far we have considered lifting conditions applicable for a monad, which can in particular be instantiated on all algebras. We now discuss a lifting condition at the level of individual algebras.

\begin{defn}
	\label{defn_indiscrete}
	Let $T$ be a monad on $\cat{Set}$. We call a $T$-algebra $(A,e)$ \emph{indiscrete} if the algebra square
 $$
 \begin{tikzcd}
	T^2 A \ar{d}[swap]{\mu} \ar{r}{Te} & TA \ar{d}[swap]{e} \\
	TA \ar{r}{e} & A
 \end{tikzcd}
 $$
 is weakly cartesian.
\end{defn}

%Note that $T$ is left reversible if and only if every free $T$-algebra is indiscrete.

In terms of partial evaluations, we will see that indiscrete algebras give partial evaluations that can always be reversed and induce the equivalence relation of having equal total evaluation (\Cref{pe_equiv}). This motivates our terminology \emph{indiscrete}.

\begin{ex}
	\label{group_indiscrete}
	%Continuing on from \Cref{monoid_lrev}, 
Let $G$ be a group. Then the $G$-sets, which are the algebras of the monad $G\times -$, are all indiscrete algebras.
 
 Indeed we show that the algebra square
 $$
 \begin{tikzcd}
	G \times G \times A \ar{r}{\id[G]\times e} \ar{d}[swap]{\mu\times \id[A]} & G\times A \ar{d}[swap]{e} \\
	G\times A \ar{r}{e} & A
 \end{tikzcd}
 $$
 is a (strong) pullback. 
 So let $(g,a)$ and $(h,b)$ in $G\times A$ be such that $e(g,a)=e(h,b)$, that is $g a = h b$. Take the element $(g,g^{-1}h,b)\in G\times G\times A$. We have 
 $$
 (\mu \times \id[A])(g,g^{-1}h,b) \;=\; (g g^{-1}h,b) \;=\; (h,b)
 $$
 and 
 $$
 (\id[G]\times e)(g,g^{-1}h,b) \;=\; (g,g^{-1}hb) \;=\; (g,g^{-1}hb) \;=\; (g,g^{-1}ga) \; = \; (g,a) .
 $$
 No other element of $G \times G \times A$ would give us the desired result: the first component must be $g$ in order to map by $G \times e$ to $(g,a)$, while the third component must be $b$ to map by $\mu \times A$ to $(h,b)$; then again because $\mu \times \id[A]$ sends our triple to $(h,b)$, the second component must be $g^{-1} h$.
\end{ex}

\begin{ex}
	\label{malcev_indiscrete}
	As we will see in \Cref{malcev_pev}, every model of a \emph{Mal'cev theory} is an indiscrete algebra of the corresponding monad. For example, since the theory of groups is a Mal'cev theory, every group is an indiscrete algebra of the group monad. Likewise, every abelian group is an indiscrete algebra of the abelian group monad.
\end{ex}

\section{Partial evaluations and their compositional properties}
 \label{secpev}

 We here recall the definition of partial evaluations together with some of their basic properties from~\cite{FP}, and we also prove a number of new results, in particular that the partial evaluation relation is the smallest relation internal to Eilenberg--Moore algebras which relates every formal expression to its result.
 
\subsection{Partial evaluations}

Following~\cite{FP}, our starting point is the simple observation that a formal expression like $3 + 4 + 5$ can not only be totally evaluated to $12$, but it can also be ``partially evaluated'' to $7 + 5$, and that the theory of monads provides a convenient framework for giving a general definition of partial evaluations. If $T$ is a monad on $\Set$ and $e : TA \to A$ is a $T$-algebra, then elements of $TA$ are formal expressions; and a formal expression $t_0 \in TA$ can be \emph{partially evaluated} to a formal expression $t_1 \in TA$ if there is $\tau \in TTA$ such that
\begin{equation}
	\label{pe_defn}
	t_0 = \mu(\tau), \qquad t_1 = (Te)(\tau).
\end{equation}
This intuitively means that $\tau$ is a doubly formal expression which results in $t_0$ upon removing the outer level of formality, and results in $t_1$ upon evaluating the inner level of formality. For the above example, we may take $T$ to be the commutative-monoid monad and
\begin{align}
	\label{first_example_boxed}
	\begin{split}
		t_0 & = \boxed{3} + \boxed{4} + \boxed{5}, \\
		t_1 & = \boxed{7} + \boxed{5}, \\
		\tau & = \boxed{\boxed{3} + \boxed{4}} + \boxed{\boxed{5}},
	\end{split}
\end{align}
where the boxings represent the levels of formality; we will explain this notation in more detail in \Cref{trees_and_boxes}.
(Note that for the commutative-monoid monad, $\boxed{x} + \boxed{y}=\boxed{y} + \boxed{x}$ on the nose, and so rearranging terms in the sum does not require a partial evaluation.)

The equations \Cref{pe_defn} can also be understood in terms of the $T$-algebra diagram
\[
	\begin{tikzcd}
		TTA \ar{r}{Te} \ar{d}[swap]{\mu}	& TA \ar{d}[swap]{e}	\\
		TA \ar{r}{e}			& A
	\end{tikzcd}
\]
which has the given elements $t_0$ and $t_1$ in the lower left and upper right corners, and the element $\tau$ lifts both of these to the upper left (whenever it exists). This makes it obvious that $e(t_0) = e(t_1)$ is a necessary condition for $t_1$ to be a partial evaluation of $t_0$.

Whenever we are only interested in the existence of a partial evaluation from $t_0$ to $t_1$, then we speak of the \emph{partial evaluation relation}. However, in this paper we will go further and in particular study properties of the \emph{partial evaluation witness} $\tau$. 

\begin{remark}
For probability monads, the partial evaluation relation has long been studied in probability theory and economics, where it is known as \emph{second-order stochastic dominance}~\cite[Section~6]{FP}.
\end{remark}

\subsection{Compatibility with $T$-algebraic structure}

If $A$ is an algebra of a monad $T$ on $\Set$ and $R \subseteq A \times A$ is a relation, then $R$ is \emph{internal} if it is a $T$-subalgebra of $A \times A$~\cite{borceux2}, where $A \times A$ carries the usual componentwise $T$-structure corresponding to the product of $T$-algebras.

\begin{defn}
 Let $f,g:X\to Y$ be functions. The \emph{relation generated by $f$ and $g$} is the relation on $Y$ given by the set-theoretical image of the pairing map $(f,g):X\to Y\times Y$.
\end{defn}

This terminology is convenient in that it allows us to say that the partial evaluation relation for a $T$-algebra $(A,e)$ is the relation generated by $\mu$ and $Te:TTA\to TA$.

\begin{prop}
	\label{pe_internal}
	Let $T$ be a monad on $\Set$ and $(A,e)$ any $T$-algebra. The partial evaluation relation on $TA$ is an internal relation, and moreover it is the smallest internal relation which relates a formal expression to its (total) result.
\end{prop}

We present the proof below based on the following technical lemma. 

\begin{lemma}\label{internalrel}
 Let $T$ be a monad on $\cat{Set}$. Then the relation generated by a pair of parallel morphisms of $T$-algebras is internal. Moreover, if $(A,e)$ is a $T$-algebra, then the smallest \emph{internal} relation larger or equal than the (set-theoretical) relation generated by a pair of maps $f,g:X\to A$ for any set $X$ is the relation generated by the parallel pair of composites
 \begin{equation}\label{newpair}
 \begin{tikzcd}
	TX \ar[shift left]{r}{Tf} \ar[shift right]{r}[swap]{Tg} & TA \ar{r}{e} & A 
 \end{tikzcd}
 \end{equation}
\end{lemma}
Note that these composites are the mates of $f$ and $g$---sometimes denoted by $f^\sharp$ and $g^\sharp$---under the usual monadic adjunction,
\[
	\Set^T(TX,A) \cong \Set(X,A).
\]
\begin{proof}[Proof of \Cref{internalrel}]
 First of all, the pairing $(p,q):A\to B\times B$ of two morphisms of algebras $p,q:A\to B$ is again a morphism of algebras. The relation generated by $p$ and $q$ is the set-theoretic image of this map, and since the forgetful functor $U:\cat{Set}^T\to \cat{Set}$ preserves image factorizations~\cite[Theorem~4.3.5]{borceux2}, it follows that this image is a $T$-subalgebra.
 
 Now let $f,g:X\to A$. The relation generated by $e\circ Tf$ and $e\circ Tg$ is internal, as we have just shown, and a straightforward argument involving $\eta : X \to TX$ shows that it contains the relation generated by $f$ and $g$.
 Suppose now that an internal relation $R\subseteq A\times A$ contains the one generated by $f$ and $g$, i.e.~that the map $(f,g):X\to A\times A$ factors through $R$. We have the commutative diagram
 \begin{equation}\label{smallestsubalg}
 \begin{tikzcd}
	TX \ar{r}[swap]{Tp} \ar[bend left=15]{rr}{T(f,g)} & TR \ar{d}[swap]{e} \ar{r}[swap]{Ti} & T(A\times A) \ar{d}[swap]{e} \\
	X \ar[bend right=15]{rr}[swap]{(f,g)} \ar{r}{p} & R \ar[hookrightarrow]{r}{i} & A\times A
 \end{tikzcd}
 \end{equation}
 where $i:R\to A \times A$ is the inclusion (which is a morphism of algebras), and $p$ is the unique map such that $(f,g)=i\circ p$.
 By commutativity of \Cref{smallestsubalg}, the map $e\circ T(f,g)$ factors through $R$, and so the relation $R$ contains the image of $e\circ T(f,g)$. 
 
 Now, the image of $e\circ T(f,g)$ is the relation generated by the pair \Cref{newpair}, since $e\circ T(f,g)= (e\times e)\circ (Tf,Tg)$. To see this, recall that the structure map of the product algebra $e:T(A\times A)\to A\times A$ is given by the composite
 $$
 \begin{tikzcd}
	T(A\times A) \ar{r}{\nabla} & TA \times TA \ar{r}{e\times e} & A \times A,
 \end{tikzcd}
 $$
 where the map $\nabla$ is the unique map which makes the following diagram commute,
 $$
 \begin{tikzcd}
	&& TA \\
	T(A\times A) \ar{urr}{T\pi_1} \uni{rr}{\nabla} \ar{drr}[swap]{T\pi_2} && TA\times TA \ar{u}[swap]{\pi_1} \ar{d}{\pi_2} \\
	&& TA
 \end{tikzcd}
 $$
 where $\pi_1,\pi_2:A\times A\to A$ are the product projections.
 Now by the commutativity of
 $$
 \begin{tikzcd}
	&& TA  \ar{r}{e} & A \\
	TX \ar[bend left]{urr}[swap]{Tf} \ar{r}{T(f,g)} \ar[bend right]{drr}{Tg}  & T(A\times A) \ar{ur}{T\pi_1} \ar{r}{\nabla} \ar{dr}[swap]{T\pi_2} & TA\times TA \ar{u}[swap]{\pi_1} \ar{d}{\pi_2} \ar{r}{e\times e} & A\times A \ar{u}[swap]{\pi_1} \ar{d}{\pi_2} \\
	&& TA  \ar{r}{e} & A
 \end{tikzcd}
 $$
 and by the universal property of the product $A\times A$, we conclude that $e\circ T(f,g) = (e\times e)\circ\nabla\circ T(f,g) = (e\circ Tf,e\circ Tg)$.
 
 Overall, we have therefore shown that $R$ contains the internal relation generated by the pair \Cref{newpair} consisting of $e \circ Tf$ and $e \circ Tg$. This relation in turn contains the relation generated by $f$ and $g$. Since $R$ was an arbitrary internal relation containing the one generated by $f$ and $g$, it follows that the relation generated by \Cref{newpair} is the smallest internal relation generated by $f$ and $g$.
\end{proof}

\begin{proof}[Proof of \Cref{pe_internal}]
	The maps $\mu,Te:TTA\to TA$ are morphisms of algebras, and so by \Cref{internalrel}, the relation they generate is internal.
	
	Consider now the parallel pair
	$$
	\begin{tikzcd}
	 TA \ar{r}[swap]{e} \ar[bend left=15]{rr}{\id} & A \ar{r}[swap]{\eta} & TA
	\end{tikzcd}
	$$
	The relation generated by these maps is the one that links a formal expression to its total result. Note that the lower map is \emph{not} a morphism of algebras in general, because $\eta$ is not. The \emph{internal} relation generated by these maps, by \Cref{internalrel} and instantiating \Cref{newpair} for the free algebra $(TA,\mu)$, is given by the pair of composites
	$$
	\begin{tikzcd}
	 TTA \ar{r}[swap]{Te} \ar[bend left=15]{rr}{\id} & TA \ar{r}[swap]{T\eta} & TTA \ar{r}{\mu} & TA
	\end{tikzcd}
	$$
	which since $\mu\circ T\eta=\id$ (right unitality triangle of the monad) is equal to 
	$$
	\begin{tikzcd}
		TTA \ar[shift left]{r}{\mu} \ar[shift right]{r}[swap]{Te} & TA,
	\end{tikzcd}
	$$
	which generates the partial evaluation relation by definition.
\end{proof}

It is a standard fact, used for example in the context of Beck's monadicity theorem, that an algebra $A$ is canonically a quotient of the free algebra $A$, in the sense of being the coequalizer of $\mu, Te : TTA \to TA$. Interpreting these maps in terms of partial evaluations gives the following.

\begin{remark}
	Every algebra $A$ is the quotient algebra of $TA$ obtained by identifying formal expressions with their results.
\end{remark}

\subsection{Composition of partial evaluations}\label{composition}

Suppose now that we have three formal expressions $t_0, t_1, t_2 \in TA$, and that $t_1$ is a partial evaluation of $t_0$ with witness $\tau_{01}$, and likewise that $t_2$ is a partial evaluation of $t_1$ with witness $\tau_{12}$. Then does it follow that $t_2$ is also a partial evaluation of $t_1$? In other words, is the partial evaluation relation transitive? And if so, is there a canonical choice of witness constructed in terms of $\tau_{01}$ and $\tau_{12}$?

In~\cite{FP}, we had shown that if $T$ is a weakly cartesian monad, then the partial evaluation relation is indeed transitive. In fact, the proof goes through in general for BC monads, and can be illustrated in terms of the following diagram.

\begin{equation}
\label{theta_cube}
\begin{split}
\begin{tikzpicture} 
\tikzset{nodestyle/.style={inner sep = 3pt, align = center}};
	\node[nodestyle]at (-1,9) (D) {$\{*\}$};
	\node[nodestyle]at (1,4) (a') {$T^2A$};
	\node[nodestyle,blue]at (0.2,4.4) (ta') {$\tau_{01}$};
	\node[nodestyle]at (5,4) (b') {$TA$};
	\node[nodestyle,blue]at (5,3.5) (tb') {$t_1$};
	\node[nodestyle]at (5,7.8) (c') {$T^2A$};
	\node[nodestyle,blue]at (5,8.3) (tc') {$\tau_{12}$};
	\node[nodestyle]at (1,7.8) (d') {$T^3A$};
	\node[nodestyle,blue]at (1,8.3) (td') {$\Theta$};
	\node[nodestyle]at (2.2,2.5) (a) {$TA$};
	\node[nodestyle,blue]at (2.2,2) (ta) {$t_0$};
	\node[nodestyle]at (6,2.5) (b) {$A$};
	\node[nodestyle]at (6,6.5) (c) {$TA$};
	\node[nodestyle,blue]at (6.6,6.5) (tc) {$t_2$};
	\node[nodestyle]at (2.2,6.5) (d) {$T^2A$};
	\draw[->](a') -- node[gray, above]{\tiny $Te$} (b');
	\draw[->](c')-- node[gray, left]{\tiny $\mu$}(b');
	\draw[->](d')-- node[gray, above]{\tiny $TTe$}(c');
	\draw[->](d')-- node[gray, left]{\tiny $\mu$}(a');
	\draw[white,line width=7, -](d) -- (c);
	\draw[white,line width=7, -](d) -- (a);
	\draw[->](a)-- node[gray, above]{\tiny $e$}(b);
	\draw[->](c)-- node[gray, right]{\tiny $e$}(b);
	\draw[->](d)-- node[gray, above]{\tiny $Te$}(c);
	\draw[->](d)-- node[gray, left, yshift=8]{\tiny $\mu$}(a);
	\draw[->](a')-- node[gray, below, near start]{\tiny $\mu\ $}(a);
	\draw[->](b')-- node[gray, above]{\tiny $\ e$}(b);
	\draw[->](c')-- node[gray, above]{\tiny $\ \ \ Te$}(c);
	\draw[->](d')-- node[gray, above, near end]{\tiny $\ \ T\mu$}(d);
	\draw[dashed,->](D) to[out=-90](a');
	\draw[thick,dotted,->](D)--(d');
	\draw[dashed,->](D) to[out=10](c');
	\draw[blue!20!white,line width=5, opacity=0.5](a) -- (b');
	\draw[blue,->](a)--(b');
	\draw[blue!20!white,line width=5, opacity=0.5](b') -- (c);
	\draw[blue,->](b')--(c);
	\draw[white,line width=7, -](a) -- (c);
	\draw[green!20!white,line width=5, opacity=0.5](a) -- (c);
	\draw[blue,->](a)--(c);
\end{tikzpicture}
\end{split}
\end{equation}
This diagram of ordinary black arrows commutes by general properties of monads. The blue arrows indicate the partial evaluations, keeping in mind that these are not morphisms in the same way as the other arrows are. Since the back square is a naturality square for $\mu$, it is possible to lift $\tau_{01} \in T^2 A$ and $\tau_{12} \in T^2 A$ to an element $\Theta \in T^3 A$ as soon as $\mu$ is a weakly cartesian transformation, which in particular holds in a BC monad. Then using commutativity of the diagram, it is easy to see that $\tau_{02} \coloneqq (T\mu)(\Theta)$ is a partial evaluation witness from $t_0$ to $t_2$. We intuitively think of the new witness $\tau_{02}$ as a \emph{composite} of the witnesses $\tau_{01}$ and $\tau_{12}$, and we therefore also call $\Theta$ a \emph{composition strategy}.

We have hence shown the following:

\begin{prop}
	\label{pe_transitive}
	If $T$ is a BC monad and $A$ any $T$-algebra, then the partial evaluation relation on $TA$ is transitive.
\end{prop}

It is natural to ask whether this transitivity holds in general. This turns out not to be the case, but finding a counterexample has been surprisingly tricky.

\begin{thm}
	\label{pe_not_transitive}
	There is a finitary monad $T$ on $\mathsf{Set}$ together with a $T$-algebra $A$ such that the partial evaluation relation on $TA$ is not transitive.	
\end{thm}

The following proof presents an explicit example. The way in which we found this example owes a lot to work of Clementino, Hofmann and Janelidze: we first constructed a semiring satisfying conditions (a)--(c) but not (d)--(e) of their \cite[Theorem 8.10]{CHJ}. However, the following presentation is largely self-contained.

\begin{proof}
	Let $S$ be the commutative semiring\footnote{Recall that a \emph{semiring} (sometimes called a \emph{rig}) is defined as a set with addition and multiplication operations like on a ring, but without the requirement that additive inverses must exist~\cite{golan}. $\N[X]$ is the semiring of polynomials in one variable $X$ with natural number coefficients and the usual addition and multiplication of polynomials.} $S \coloneqq \N[X]/\langle X^2 = 2\rangle$. This means that the elements of $S$ are of the form\footnote{If so desired, we could also make $S$ finite by imposing $2 + 1 = 2$ in addition, so that every semiring element could be represented as above with $a,b\in\{0,1,2\}$. This would result in $|S| = 9$, and the same argument would go through and produce a more minimal counterexample. But we will not do this in order to keep the example as simple as possible.}
	\[
		a + b X
	\]
	for $a,b\in\N$, with componentwise addition, and multiplication such that
	\[
		(a_1 + b_1 X) (a_2 + b_2 X) = (a_1 a_2 + 2 b_1 b_2) + (a_1 b_1 + a_2 b_2) X.
	\]
	This semiring can be realized concretely as the smallest subsemiring of $(\R_+,+,\cdot)$ containing the number $\sqrt{2}$, so that $a + bX$ corresponds to $a + b\sqrt{2}$.\footnote{We thank Martti Karvonen for having pointed this out to us.}

	The equation
	\begin{equation}
		\label{XX2}
		X \cdot X = 1 + 1
	\end{equation}
	in $S$ will be what makes the counterexample work, together with the following two facts:
	\begin{itemize}
		\item $X\in S$ is additively indecomposable: if $X = r + s$ with $r,s\in S$, then $r = 0$ or $s = 0$.
		\item There is no $r\in S$ with $Xr = 1$: writing $r = a + bX$ for a putative such $r$, we find that $Xr$ must have constant coefficient $\geq 2$ if $b \neq 0$, which we do not want; but if $b = 0$, then $Xr$ is merely a multiple of $X$, which is not what we want either.
	\end{itemize}
	Now let $T$ be the $S$-semimodule monad. This means that $TX$ for $X \in \Set$ is the set of finitely supported functions $X \to S$, and we interpret and denote these as formal $S$-linear combinations. The monad structure of $T$ is the obvious one which makes $T$-algebras into $S$-semimodules; we refer to~\cite[Section~6]{CHJ} for more details.
	
	Let $A \coloneqq \{\ast\}$ the one element $T$-algebra, i.e.~the zero $S$-semimodule. We will use box notation as in \Cref{first_example_boxed}. 
	Now, $\boxed{\ast}\in TA$ partially evaluates to $2\,\boxed{\ast}\in TA$, as witnessed by
			\[
				\boxed{\bullet} + \boxed{\boxed{\ast}} \in TTA.
			\]
			Indeed this is the formal $S$-linear combination given by the formal sum of the ``empty expression'' $0 = \bullet \in TA$ and $\eta{\ast} = \boxed{\ast} \in TA$. Applying $\mu$ to this doubly formal expression removes the outer brackets, giving $0 + \boxed{\ast} = \boxed{\ast}$, while applying $Te$ amounts to removing the inner brackets, including the evaluation of $\bullet$ to $\ast$, giving the desired $\boxed{\ast} + \boxed{\ast}$.
	Also, $2\, \boxed{\ast}\in TA$ partially evaluates to $X\boxed{\ast}\in TA$, as witnessed by
			\[
				X\boxed{X \boxed{\ast}} \in TTA.	
			\]
			Indeed, removing the outer brackets gives $X^2\, \boxed{\ast} = 2\, \boxed{\ast}$, while removing the inner brackets results in $X\boxed{\ast}$ since $X\cdot \ast = \ast$.
	Thus if the partial evaluation relation were transitive, $\boxed{\ast}\in TA$ would also have to partially evaluate to $X\boxed{\ast}\in TA$.
	To see that this is not the case, note that elements $\tau \in TTA$ are finitely supported functions $\tau : TA \to S$. Using the definition of the functor $T$ on the algebra map $e$ gives a description of $(Te)(\tau) \in TA$ as a finitely supported function $A \to S$, namely
			\[
				(Te)(\tau) = \left( \ast \mapsto \sum_{t\in TA} \tau(t) \right) .
			\]
			In other words, since $A$ is the zero module, $Te : TTA \to TA$ simply sums up all values of the function $\tau$.
			
			Now suppose that such a $\tau : TA \to S$ witnesses the putative partial evaluation from $\boxed{\ast}$ to $X\boxed{\ast}$. Since the sum of values of $\tau$ must be $X$, by using the fact that $X$ is additively indecomposable in $S$, we conclude that we must have $\tau(t) = X \delta_{t,r}$ for some $r\in TA$.

			On the other hand, applying $\mu$ to this $\tau$ then results in
			\[
				\sum_t \tau(t) \, t\cdot \boxed{\ast} = X r\cdot \boxed{\ast}.
			\]
			In order for this to be equal to just $\boxed{\ast}$, we need to have $Xr = 1$. But this is impossible in $S$ as also noted above.\qedhere
\end{proof}

\subsection{Reversing partial evaluations}\label{reversing}

For the commutative-monoid monad, there is a partial evaluation from $\boxed{3}+\boxed{4}+\boxed{5}$ to $\boxed{7}+\boxed{5}$, but there is none the other way around. Thus, as its name already indicates, the partial evaluation relation is typically not symmetric. However, there also are monads for which it is symmetric on \emph{all} of its algebras. In the following, we give some criteria for when this and related phenomena occur. Recall the notion of indiscrete algebra from \Cref{defn_indiscrete}.

\begin{prop}
	\label{pe_equiv}
	Let $T$ be a monad on $\cat{Set}$, and let $(A,e)$ be a $T$-algebra. Then $A$ is indiscrete if and only if the partial evaluation relation is an equivalence relation. In this case, the equivalence relation obtained is the kernel pair of $e : TA \to A$.
\end{prop}

The final statement means that for $t_0, t_1 \in TA$, there is a partial evaluation from $t_0$ to $t_1$ if and only if these two expressions have the same result, $e(t_0) = e(t_1)$. Equivalently, the quotient of the equivalence relation is exactly $A$.

\begin{proof}
	Suppose that $(A,e)$ is indiscrete, meaning that the algebra square
 	$$
 		\begin{tikzcd}
			TTA \ar{d}[swap]{\mu} \ar{r}{Te} & TA \ar{d}[swap]{e} \\
	  		TA \ar{r}{e} & A
 		\end{tikzcd}
	$$
	is a weak pullback. This means exactly that two formal expressions in $TA$ have the same result in $A$ if and only if there exists a partial evaluation between them.

	Conversely, suppose that the partial evaluation relation for $A$ is an equivalence relation. Since two expressions that have different results cannot admit a partial evaluation between them, the partial evaluation relation must be finer or equal than the kernel pair of $e$. Just as well, since there always exists a partial evaluation from any formal expression to its result (total evaluation), the equivalence relation is necessarily coarser or equal to the kernel pair of $e$. Hence the two equivalence relations coincide.
\end{proof}

\begin{ex}
	\label{group_equiv}
	Let $G$ be a group and $G \times -$ the $G$-action monad. Then the algebras of this monad are indiscrete (\Cref{group_indiscrete}), and hence the partial evaluation relation for $G$-sets is an equivalence relation. More concretely, for a $G$-set $A$ it is the relation on $G \times A$ given by $(g,a) \sim (h,b)$ if and only if $ga = hb$.
\end{ex}

The same turns out to be the case for the group monad and the abelian group monad, where one can also intuitively ``invert'' things. These are instances of a more general statement which we now turn to, based on the following classical notion of universal and categorical algebra~\cite{borceuxbourn}.

\begin{defn}
 A \emph{Mal'cev operation} on a set $A$ is a ternary operation $m:A\times A\times A\to A$ such that for each $a,b\in A$,
 $$
 m(a,b,b) \;= \; a \qquad \mbox{and} \qquad m(a,a,b) \;=\; b .
 $$
 A \emph{Mal'cev theory} is an algebraic theory which contains a Mal'cev operation. 
\end{defn}

In the theory of groups, there is a Mal'cev operation given by 
$$
m(a,b,c) \; \coloneqq \; a\,b^{-1}\,c.
$$
Therefore any theory whose algebras are groups, with possibly extra structures or properties, is a Mal'cev theory. This includes the theories of groups, abelian groups, rings, commutative rings, and modules over a fixed ring,
but not, for example, the theory of monoids, commutative monoids, semirings, and semimodules over a semiring which is not a ring. 
An example of a Mal'cev theory which is not a theory of particular groups is the theory of \emph{heaps}, closely related to \emph{torsors} (see for example \cite{heaps}).

The following well-known statement is why we are interested in Mal'cev theories, together with the usual correspondence between models of an algebraic theory and the algebras of the associated monad.

\begin{prop}[{e.g.~\cite[Chapter~2]{borceuxbourn}}]
	\label{malceveq}
	An algebraic theory is Mal'cev if and only if every internal reflexive relation in the category of $T$-algebras is an equivalence relation. 
\end{prop}

Recall that an \emph{internal} relation is one which is compatible with the algebraic operations, or equivalently one in which the relation itself is a model of the theory (or equivalently a $T$-algebra).

Since the partial evaluation relation is an internal reflexive relation (\Cref{pe_internal}), we therefore obtain the following by \Cref{pe_equiv}.

\begin{cor}
	\label{malcev_pev}
	Let $T$ be the monad on $\Set$ associated to a Mal'cev theory. Then every $T$-algebra $A$ is indiscrete, and the partial evaluation relation on $TA$ is an equivalence relation.
\end{cor}

Note that the converse is not true: for $G$ a group, the theory of $G$-actions is not Mal'cev, since there is no operation of arity two or higher, but the partial evaluation relation is still an equivalence relation (\Cref{group_equiv}).

\subsection{Irreversibility of partial evaluations}\label{irreversibility}

Finally, we consider some conditions on the monad and the algebra which amount to a certain kind of irreversibility of partial evaluations. Recall from \Cref{strictly_pos} that for any algebra $A$ of a strictly positive monad $T$ the square
\[
 \begin{tikzcd}
	A \ar[equals]{d}{} \ar{r}{\eta\eta} & TTA \ar{d}[swap]{\mu} \\
	A \ar{r}{\eta} & TA. 
 \end{tikzcd}
\]
is a pullback. This also has some significance for partial evaluations.

\begin{prop}
	\label{total_terminal}
	A monad $T$ on $\Set$ is strictly positive if and only if for any $T$-algebra $(A,e)$ and $a \in A$, 
	the term $\eta(a)\in TA$ can only be partially evaluted to itself, and the only possible witness is $\eta\eta(a) \in TTA$.
\end{prop}

\begin{proof}
	Straightforward unfolding of definitions.
\end{proof}

Note that the element $\eta\eta(a) \in TTA$ is the canonical reflexivity witness of the partial evaluation from $\eta(a)$ to itself. In particular, \ref{total_terminal} implies that if $\eta(a)$ partially evaluates to any $t \in TA$, then $t = \eta(a)$.

If $T$ is strictly positive and the functor part preserves weak pullbacks, then the following diagram is also a pullback.
 $$
 \begin{tikzcd}
	TA \ar[equals]{d}{} \ar{r}{T(\eta\eta)} & T^3 A \ar{d}[swap]{T\mu} \\
	TA \ar{r}{T\eta} & TTA. 
 \end{tikzcd}
 $$
 In terms of partial evaluations, this square being a pullback means exactly that the identity partial evaluation cannot be expressed as a nontrivial composite (i.e.~it can only be written as the composition of twice itself). To see this, recall that given a composition strategy $\Theta\in T^3 A$, the resulting composite partial evaluation witness is given by $(T\mu)(\Theta)$. Now suppose that for some $t \in TA$ we have $(T\mu)(\Theta) = (T\eta)(t)$, i.e.~the identity partial evaluation at $t$ arises in this way from the composition strategy $\Theta$. Then the pullback condition says that necessarily $\Theta = \left(T(\eta\eta)\right) (\alpha)$.

%%% T: the following seems to fit in better with the later sections, so I've commented it out for now and we can possibly move it to there later
% 
% We can decompose $T(\eta\eta)$ as the commutative diagram
%$$
%\begin{tikzcd}
% TA \ar{r}{T\eta} & TTA \ar[shift left]{r}{T\eta} \ar[shift right]{r}[swap]{TT\eta} & T^3 A ,
%\end{tikzcd}
%$$
%which commutes by naturality of $\eta$. We see that $\Theta$ is exactly the degenerate 2-simplex obtained from the identity 1-simplex of $\alpha$ (and it can be obtained as the degenerate 2-simplex either on the left or on the right). 

 Finally, the following result shows that strict positivity of a monad and indiscreteness of an algebra rarely come together.

\begin{prop}\label{positive_indiscrete}
	Let $T$ be a strictly positive monad on $\Set$ and let $(A,e)$ be an indiscrete $T$-algebra. Then the partial evaluation relation of $A$ is the identity relation on $TA$ and $e:TA\to A$ is a bijection.
\end{prop}

\begin{proof}
	Let $t\in TA$ be a formal expression, and let $a \coloneqq e(t)$ be its total result. By indiscreteness there exists not only a partial evaluation from $t$ to $\eta(a)$, but also one from $\eta(a)$ to $t$. By strict positivity, the partial evaluation from $\eta(a)$ to $t$ must be an identity, which means that $t=\eta(a)$. This is true for all $t\in TA$, which means in particular that $e$ is injective, and hence (since it is split epi) an isomorphism.
\end{proof}

% \Cref{fully_complete_discrete} strengthens this result to all right reversible algebras and to more general simplicial sets than those arising as bar constructions.

\section{The bar construction and the quest for its compositional structure}\label{secbar}

Consider again the diagram~\Cref{theta_cube} involving the composition strategy $\Theta$. The three blue arrows which illustrate the partial evaluations indicate that it may be beneficial to think of $\Theta$ as a \emph{triangle} or \emph{2-simplex} in a structure where the elements of $TA$ are vertices, the elements of $T^2 A$ are edges between these vertices representing partial evaluation witnesses, the 2-simplices are composition strategies, etc. We do not need to look far in order to obtain a general definition for what this structure is, since it is well known: the bar construction. We refer to Trimble's exposition~\cite{trimble} for a more extensive treatment of the bar construction and its categorical properties.

\subsection{The bar construction}\label{barconstruction}

Given a monad $(T,\mu,\eta)$ on $\Set$ and a $T$-algebra $(A,e)$, the bar construction gives a free resolution of that algebra in the form of an augmented simplicial set, i.e.~a functor 
\[
	\barc{T}{A} \: : \: \Delta_+^\op \longrightarrow \Set,
\]
defined as follows. On objects, 
\[
	\barc{T}{A}(n) \coloneqq T^{n+1} A
\]
for all ${[n]} \in \Delta_+$, including $n = -1$. Thus an $n$-dimensional simplex in the bar construction is an element of $T^{n+1} A$, i.e.~a formal expression with elements from $A$ and $n+1$ levels of formality. The generating face maps
\[
	d_{n,i} \: : \: T^{n+1} A \longrightarrow T^n A,
\]
are given by $T^n e$ for $i=0$ and by $T^{n-i}\mu$ for $1 \leq i\leq n$, resulting in the diagram
$$
\begin{tikzcd}
	\cdots \quad T^4 A \arrow[shift left=7.5]{r}{T^3 e} \arrow[shift left=2.5]{r}{T^2 \mu} \arrow[shift right=2.5]{r}{T \mu} \arrow[shift right=7.5]{r}{\mu} & T^3 A \arrow[shift left=5]{r}{T^2 e} \arrow{r}{T \mu} \arrow[shift right=5]{r}{\mu} & T^2 A \arrow[shift left=2.5]{r}{T e} \arrow[shift right=2.5]{r}{\mu} & T A \arrow{r}{e} & A.
\end{tikzcd}
$$
The degeneracy maps
\[
	s_{n,i} \: : \: T^{n+1}A \longrightarrow T^{n+2}A
\]
are given by $T^{n-i+1}\eta$ for $0\leq i\leq n$, resulting in the diagram
$$
\begin{tikzcd}
	\cdots \quad T^4 A  & T^3 A \arrow[shift left=5,swap]{l}{T^{3}\eta} \arrow[swap]{l}{T^2\eta} \arrow[shift right=5,swap]{l}{T\eta} & T^2 A \arrow[shift left=2.5,swap]{l}{T^2\eta} \arrow[shift right=2.5,swap]{l}{T\eta} & T A \arrow[swap]{l}{T\eta} & A.
\end{tikzcd}
$$
Taken together, the face and degeneracy maps define the augmented simplicial set $\barc{T}{A} : \Delta_+^\op \to \Set$. Its restriction to $\Delta$ is a simplicial set which we also denote $\barc{T}{A}$ by abuse of notation; throughout the paper $\barc{T}{A}$ will refer to the latter since the augmentation plays no role for us.

\begin{remark}
	\label{cartesian_nerve}
	It is well-known that if $T$ is a cartesian monad, then $\barc{T}{A}$ is the nerve of a category.\footnote{This seems to be a folklore observation for which the earliest occurrence that we know of is a comment by Trimble on the n-Category Caf\'e blog at \href{https://golem.ph.utexas.edu/category/2007/05/on_the_bar_construction.html\#c009955}{https://golem.ph.utexas.edu/category/2007/05/on{\textunderscore}the{\textunderscore}bar{\textunderscore}construction.html\#c009955}. See also~\cite[Proposition~4.4.1]{weber_classifiers} for a more general statement.}

	Indeed if we assume merely that $\mu$ is strongly cartesian, then the following diagram already proves the Segal condition $X_n \cong X_1 \times_{X_0} \scdots{n} \times_{X_0} X_1$ for $X = \barc{T}{A}$, making it into the nerve of a category.
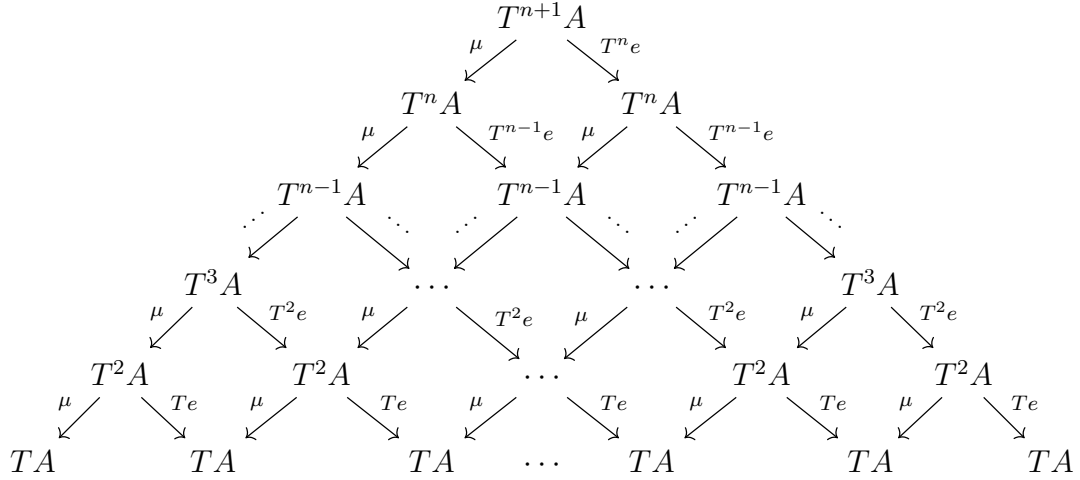
\begin{center}\begin{tikzcd}[column sep=.1cm,row sep=scriptsize]
	 & & & & & T^{n+1}A \arrow{dl}[swap]{\mu} \arrow{dr}{T^n e} & & & & & \\
	 & & & & T^n A \arrow{dl}[swap]{\mu} \arrow{dr}{T^{n-1} e} & & T^n A \arrow{dl}[swap]{\mu} \arrow{dr}{T^{n-1} e} & & & & \\
	 & & & T^{n-1} A \arrow{dl}[swap]{\iddots} \arrow{dr}{\ddots} & & T^{n-1} A \arrow{dl}[swap]{\iddots} \arrow{dr}{\ddots} & & T^{n-1} A \arrow{dl}[swap]{\iddots} \arrow{dr}{\ddots} & & & \\
	 & & T^3 A \arrow{dl}[swap]{\mu} \arrow{dr}{T^2 e} & & \cdots \arrow{dl}[swap]{\mu} \arrow{dr}{T^2 e} & & \cdots \arrow{dl}[swap]{\mu} \arrow{dr}{T^2 e} & &  T^3 A \arrow{dl}[swap]{\mu} \arrow{dr}{T^2 e} & & \\
	 & T^2 A \arrow{dl}[swap]{\mu} \arrow{dr}{T e} & & T^2 A \arrow{dl}[swap]{\mu} \arrow{dr}{T e} & & \cdots \arrow{dl}[swap]{\mu} \arrow{dr}{T e} & & T^2 A \arrow{dl}[swap]{\mu} \arrow{dr}{T e} & & T^2 A \arrow{dl}[swap]{\mu} \arrow{dr}{T e} & \\
	 TA & & TA & & TA & \cdots & TA & & TA & & TA \\
\end{tikzcd}\end{center} 
Each square in the diagram is a naturality square for $\mu$, hence a pullback, and any cone over the bottom two rows of the diagram induces unique maps to each subsequent row moving upwards. $\barc{T}{A}$ is thus the nerve of a category for any monad $T$ with $\mu$ strongly cartesian.
\end{remark}

\begin{ex}
	\label{monoid_action_bar}
	More concretely, for a monoid $M$ let $M \times -$ be the $M$-set monad. For an $M$-set $A$, the bar construction $\barc{M \times -}{A}$ is the nerve of a category. A straightforward unfolding of the definition of the first few levels of the bar construction shows that this category has pairs $(x,a) \in M \times A$ as objects, with morphisms $(x,a) \to (y,b)$ corresponding to the monoid elements $z \in M$ satisfying $x = yz$ and $b = za$, and composing by multiplication.
	% In particular, $\barc{M\times -}{1}$ is the nerve of $M$ itself, when considered as a category with a single object. 
\end{ex}

\begin{ex}
	\label{terminal_monoid_bar}
	Let $T$ be the monoid monad and $1$ the trivial one-element monoid. Then $\barc{T}{1}$ is the nerve of the augmented simplex category $\Delta_+$, for example by~\cite[Corollary~7.2.1]{batanin}.
\end{ex}

\begin{remark}
	If $T$ is a cartesian monad and $A$ a $T$-algebra with cartesian algebra square, then $\barc{T}{A}$ is even the nerve of an equivalence relation, namely of the kernel pair of $e : TA \to A$ as in \Cref{pe_equiv}.
 This is because by definition, a $T$-algebra $(A,e)$ is cartesian if and only if the parallel pair
$$
\begin{tikzcd}
 TTA \ar[shift left]{r}{\mu} \ar[shift right]{r}[swap]{Te} & TA 
\end{tikzcd}
$$
is a kernel pair of $e$. 
	% On the other hand, it is easy to give examples in which the category of which $\barc{T}{A}$ is the nerve does have distinct parallel arrows. This happens for example for $\barc{M\times -}{1}$, where $M$ is a suitable monoid and the singleton set $1$ carries the unique $M$-action, since then the relevant category is given by the elements $a,b \in M$ as objects and such that the morphisms $a \to b$ are in bijection with the elements $c \in M$ satisfying $a = cb$.
\end{remark}

Based on the definition of the bar construction, it is immediate that a 1-simplex in $\barc{T}{A}$ from a vertex $t_0 \in TA$ to a vertex $t_1 \in TA$ is the same thing as a partial evaluation witness from $t_0$ to $t_1$. Moreover for $t_0, t_1, t_2 \in TA$ and partial evaluation witnesses $\tau_{01}, \tau_{12} \in TTA$, the composition strategies $\Theta$ (\Cref{composition}) are exactly those 2-simplices in $\barc{T}{A}$ whose face obtained by deleting vertex 2 is $\tau_{01}$, and whose face obtained by deleting vertex 0 is $\tau_{12}$. In the parlance of quasicategory theory, finding such a $\Theta$ for given $\tau_{01}$ and $\tau_{12}$ hence amounts to \emph{filling an inner 2-horn}~\cite{joyal}.

This motivates our quest of trying to understand to what extent the bar construction, considered as a simplicial set, can be thought of as a higher compositional structure. As we have seen in the previous section, all inner 2-horns can be filled if $T$ is a BC monad; on the other hand, for a general monad $T$ and two composable 1-simplices in the bar construction of a $T$-algebra, there may not even be a third 1-simplex pointing directly from the source of the first to the target of the second (\Cref{pe_not_transitive}). This theme will continue throughout the rest of the paper: we will find good compositionality properties for the bar construction as long as $T$ satisfies suitable lifting conditions, but not in general.

\subsection{Notation for higher simplices in the bar construction}
\label{trees_and_boxes}

With the commutative-monoid monad on $\Set$ serving as a recurring example in what follows, we now give a more explicit description of its bar construction in some detail, in particular making precise the idea that higher levels of formality correspond to iterated ``bracketing'' or ``boxing'' of expressions. Although this description applies in very much the same way to all monads coming from symmetric operads (and similarly from non-symmetric operads), we focus on the commutative-monoid monad for simplicity, leaving the treatment of other cases to the reader.

Variants of the following considerations are very well known in operad theory. Nevertheless, we surprisingly have not found any reference containing the relevant statements in the precise form that we need, and we therefore offer our own detailed exposition in what follows. Similar considerations in somewhat different contexts can be found e.g.~in works of Ching~\cite[Section~4]{ching} or Kock~\cite[Section~2.3.5]{kock}.

Let $T : \Set \to \Set$ be the commutative-monoid monad. For any $X \in \Set$, the set $TX$ is the set of finite multisubsets of $X$, or equivalently of finitely supported maps $X \to \N$. We can thus identify the elements of $TX$ with non-planar rooted trees of height at most $1$, where the leaves are labelled by elements of $X$. The tree consisting only of the root then corresponds to the neutral element $0 \in TX$ representing the empty multiset.

Now upon applying $T$ multiple times, it follows that $T^n X$ for $n \in \N$ can be identified with the set of non-planar rooted trees of height at most $n$ and with leaves at depth $n$ labelled by elements of $X$; see also the literature on \emph{operadic trees} for more explanation~\cite[Section~1.5]{kock}. For example for $n = 2$, a typical element of $T^2 X$ is represented by a tree that looks like
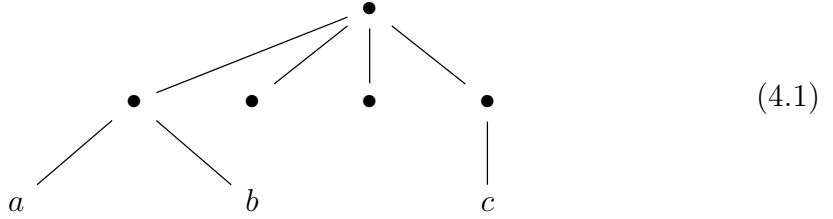
\begin{equation}
	\label{tree_example}
	\begin{tikzcd}
		& & & \bullet \ar[-]{dll} \ar[-]{dl} \ar[-]{d} \ar[-]{dr} \\
		& \bullet \ar[-]{dl} \ar[-]{dr} & \bullet & \bullet & \bullet \ar[-]{d} \\
		a & & b & & c
	\end{tikzcd}
\end{equation}
for some $a,b,c \in X$. In terms of multiset notation, we could also denote this element of $T^2 X$ as
\[
	\{ \{a, b\}, \{ \}, \{ \}, \{c\}\}.
\]
But since this type of expression can get cumbersome to write, we equivalently use boxed expressions such as
\[
	\boxed{\boxed{a} + \boxed{b}} + \boxed{\bullet} + \boxed{\bullet} + \boxed{\boxed{c}},
\]
were the tree structure is now encoded in the boxing, so that each level of boxing represents a level of formality. The fact that the trees are non-planar now means that it is understood that the individual summands within a box can be arranged arbitrarily but not moved across box boundaries. We denote unlabelled leaves by $\bullet$, as in the tree diagrams. These correspond to the neutral element $0 \in T^k X$ when there are $n - k$ levels of boxing around $\bullet$. The elements of $T^n X$ hence are identified with equivalence classes---with respect to the commutative monoid laws---of boxed expressions with up to $n$ levels of boxing and elements of $X$ at exactly level $n$. In particular, elements of $TX$ already have one level of boxing.

In terms of the trees or boxed expressions picture, we then have the following:
\begin{itemize}
	\item For a function $f : X \to Y$, the induced map $T^n f : T^n X \to T^n Y$ takes such a tree with bottom leaves labelled in $X$ and replaces these labels by applying $f$ to each of them, and similarly in the boxed expressions notation.

	\item The multiplication $\mu : T^2 X \to T X$, which takes a multiset of multisets and maps it to their multiset union, takes a tree of height at most $2$ and removes all nodes at level $1$, leaving their child nodes in place and connecting them up directly with the root node. More generally, $T^k \mu : T^{n+1}X \to T^n X$ for $k = 0,\ldots,n-1$, takes a tree of height at most $n+1$, removes all nodes at depth $k+1$, and similarly connects their children to their parent nodes.

	\item The unit $\eta : X \to TX$ takes an element of $X$ to the corresponding singleton multiset $\{x\}$. Hence in the tree picture, $T^k \eta : T^n X \to T^{n+1} X$ for $k = 0,\ldots,n$ replaces every node at depth $k$ by a pair of nodes, one at depth $k$ and one at depth $k+1$, such that the latter is the only child node of the former.
\end{itemize}

We thus have all the tools in hand to do concrete computations in the bar construction of the commutative-monoid monad: we perform them on the corresponding trees, while usually using boxed expression notation for these trees.

\begin{remark}
	\label{terms_decrease}
	The commutative-monoid monad has the following convenient property, which is obvious from the trees picture: if $t_0 \in TA$ partially evaluates to $t_1 \in TA$, then the number of terms in $t_1$ is at least as large as the number of terms in $t_0$, with equality if and only if $t_0 = t_1$. (Here, the number of terms can be conveniently defined by applying the functor $T$ to the map $A \to 1$ and composing with the obvious isomorphism $T1 \cong \N$.)
\end{remark}

\subsection{Nonuniqueness of composite partial evaluations}

Our first question concerns the uniqueness of composition strategies. Upon composing two partial evaluation witnesses using composition strategies, is the resulting composite partial evaluation witness well defined, i.e.~independent of the choice of composition strategy? Equivalently, is $\barc{T}{A}$ such that fillers for inner $2$-horns have unique third faces? This is not the case:

\begin{thm}\label{nonuniqueness}
	There is a finitary weakly cartesian monad $T$ on $\mathsf{Set}$ together with a $T$-algebra $A$ for which $\barc{T}{A}$ contains an inner $2$-horn with two different fillers such that their outer $1$-faces are also different.
\end{thm}

Note that this is a phenomenon which cannot occur in the nerve of a category.

\begin{proof}
	Again we construct a concrete example, this time with the commutative-monoid monad $T$ and the $T$-algebra $A \coloneqq (\N,+)$. Consider the elements of $TTA$ given by
	\begin{align*}
 		\alpha \;& \coloneqq\; \boxed{ \boxed{2} + \boxed{2} } 
	  	+ \boxed{ \boxed{3} + \boxed{3} }
		+ \boxed{ \boxed{3} + \boxed{1} }, \\
		\beta \;&\coloneqq\; \boxed{\boxed{4} + \boxed{6}} + \boxed{\boxed{4}} . 
	\end{align*}
	These form an inner $2$-horn, because
	$$
	(Te)(\alpha) = \mu(\beta) = \boxed{4} + \boxed{6} + \boxed{4} .
	$$
	This horn admits two distinct fillers given by
	\begin{align*}
		 \delta \;&\coloneqq\; \boxed{ \boxed{ \boxed{2} + \boxed{2} } 
		  + \boxed{ \boxed{3} + \boxed{3} } }
		  + \boxed{ \boxed{ \boxed{3} + \boxed{1} } }, \\[8pt]
		 \delta' \;&\coloneqq\; \boxed{ \boxed{ \boxed{2} + \boxed{2} } }
		  + \boxed{ \boxed{ \boxed{3} + \boxed{3} } 
		  + \boxed{ \boxed{3} + \boxed{1} } }.
	\end{align*}
	Indeed, removing the outer boxes easily gives $\mu(\delta)=\mu(\delta')=\alpha$, while removing the inner boxes shows that $(T^2e)(\delta)=(T^2e)(\delta')=\beta$, proving that both $\delta$ and $\delta'$ fill the horn. The resulting outer 1-faces arise by removing the intermediate level of boxing,
	\begin{align*}
		 (T\mu)(\delta) \;&=\; \boxed{  \boxed{2} + \boxed{2} 
		  +  \boxed{3} + \boxed{3}  }
		  + \boxed{  \boxed{3} + \boxed{1}  }, \\[8pt]
		 (T\mu)(\delta') \;&=\; \boxed{  \boxed{2} + \boxed{2}  }
		  + \boxed{  \boxed{3} + \boxed{3} 
		  +  \boxed{3} + \boxed{1}  },
	\end{align*}
	which are indeed distinct parallel 1-cells. 
\end{proof}

% By the way, by adding weights this can be turned in to an analogous statement for the distribution monad, and the resulting dilations are different. 

\subsection{Non-fillable inner horns}
\label{unfillable_horns}

Since inner 2-horns in the bar construction have fillers for BC monads, it is natural to ask whether inner horns have fillers in general under suitable assumptions on the monad. By \Cref{cartesian_nerve}, this is clearly the case for $\barc{T}{A}$ whenever $T$ is a cartesian monad since the nerve of a category trivially has all inner horn fillers. In the following, we will consider inner 3-horns; these come in the following two kinds.

An inner 3-horn of the first kind in a bar construction consists of three 2-simplices $\alpha,\gamma,\delta\in T^3A$ satisfying the equations
\begin{equation}
	\label{3horn1eq}
	(T\mu)(\alpha)=\mu(\gamma),\qquad (T^2e)(\alpha)=\mu(\delta),\qquad (T^2e)(\gamma)=(T^2e)(\delta).
\end{equation}
A filler is then an element $\varepsilon\in T^4A$ which recovers the given 2-simplices via
\[
	\alpha = \mu(\varepsilon), \qquad \gamma = (T^2\mu)(\varepsilon), \qquad \delta = (T^3e)(\varepsilon).
\]
In other words, given compatible elements $\alpha,\gamma,\delta$ of the red part of the diagram
\begin{equation}
\label{3horn1}
	\begin{tikzcd}
		\textcolor{blue}{(\varepsilon \in)\: T^4 A} \ar[blue]{rr}{T^3e} \ar[blue]{dd}[swap]{\mu} \ar[blue]{dr}[swap]{T^2\mu} &	& \textcolor{red}{(\delta\in)\: T^3A} \ar[red]{dd}[near end]{\mu} \ar[red]{dr}{T^2e} \\
		& \textcolor{red}{(\gamma \in)\: T^3A} \ar[red,crossing over]{rr}[near end, swap]{T^2e} &	& T^2 A \ar{dd}{\mu}	\\
		\textcolor{red}{(\alpha\in)\: T^3A} \ar[red]{dr}[swap]{T\mu} \ar[red]{rr}[swap, near end]{T^2e} &		& T^2 A \ar{dr}{Te}	\\
		& T^2 A \ar{rr}[swap]{Te}	\ar[red,crossing over,from=uu,near end, "\mu"]&			& TA
	\end{tikzcd}
\end{equation}
there is a common lift $\varepsilon$ along the blue arrows.

An inner 3-horn of the second kind consists of three 2-simplices $\alpha,\beta,\delta\in T^3A$ satisfying the equations
\[
	\mu(\alpha)=\mu(\beta), \qquad (T^2e)(\alpha)=\mu(\delta), \qquad (T^2e)(\beta)=(T\mu)(\delta).
\]
A filler is then an element $\varepsilon\in T^4A$ which recovers the given 2-simplices via
\[
	\mu(\varepsilon)=\alpha, \qquad (T\mu)(\varepsilon)=\beta, \qquad (T^3e)(\varepsilon)=\delta.
\]
In other words, given compatible elements $\alpha,\beta,\delta$ of the red part of the diagram
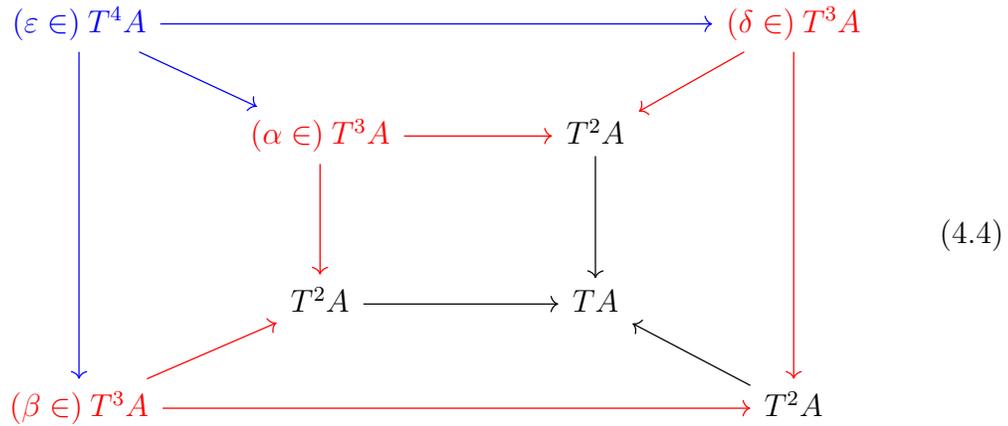
\begin{equation}
	\label{3horn2}
	\begin{tikzcd}
		\textcolor{blue}{(\varepsilon \in)\: T^4 A} \ar[blue]{rrrr}{T^3e} \ar[blue]{dddd}[swap]{T\mu} \ar[blue]{dr}[swap]{\mu} & & & & \textcolor{red}{(\delta\in)\: T^3A} \ar[red]{dl}{\mu} \ar[red]{dddd}{T\mu} \\
		& \textcolor{red}{(\alpha\in)\: T^3A} \ar[red]{dd}{\mu} \ar[red]{rr}{T^2e} & & T^2 A \ar{dd}{\mu}	\\
		\\
		& T^2 A \ar{rr}{Te} & & TA \\
		\textcolor{red}{(\beta\in)\: T^3A} \ar[red]{rrrr}[swap]{T^2e} \ar[red]{ur}{\mu} & & & & T^2 A \ar{ul}[swap]{\mu} \\
	\end{tikzcd}
\end{equation}
there is a common lift $\varepsilon$ along the blue arrows. Here, we have drawn the above diagrams in this form since they both appear in roughly their respective shape as subdiagrams of the hypercube
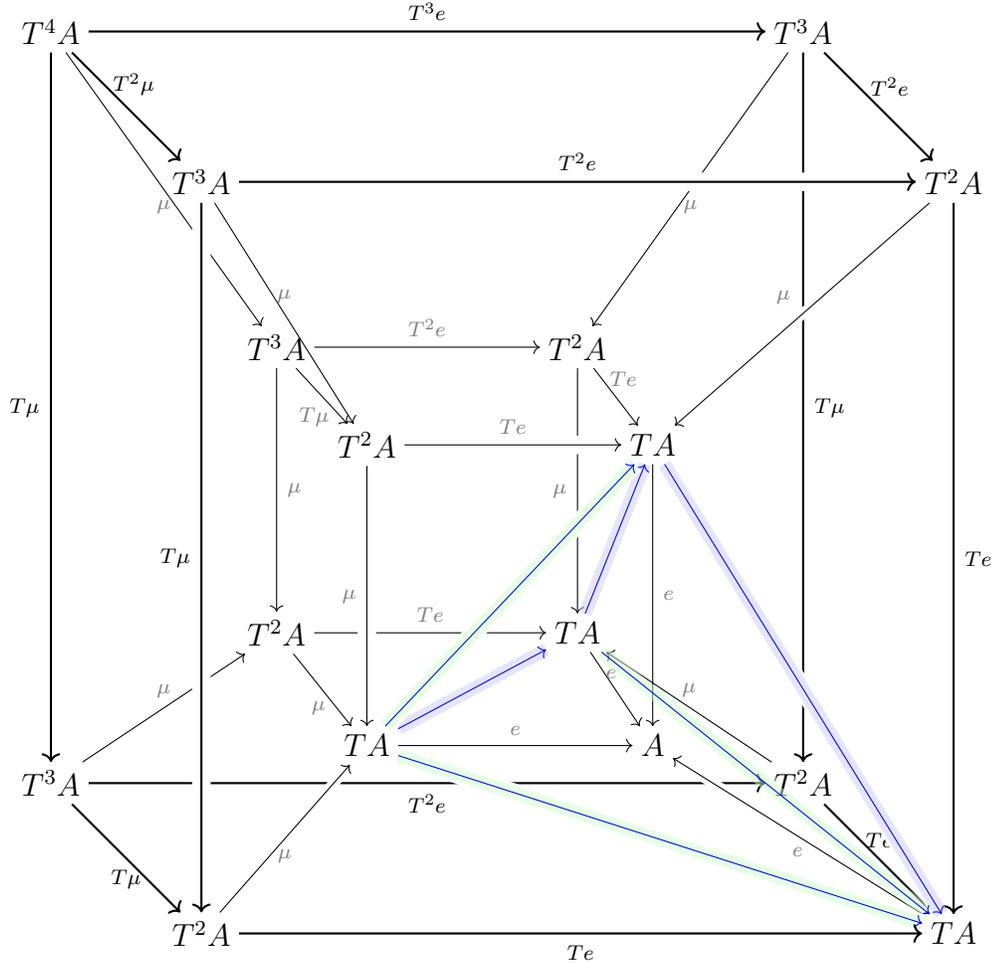
\[
	\begin{tikzpicture}
		\tikzset{nodestyle/.style={inner sep = 3pt, align = center}};
		\node[nodestyle]at (0,0) (A) {$T^2A$};
		\node[nodestyle]at (10,0) (B) {$TA$};
		\node[nodestyle]at (10,10) (C) {$T^2A$};
		\node[nodestyle]at (0,10) (D) {$T^3A$};
		\node[nodestyle]at (-2,2) (A') {$T^3A$};
		\node[nodestyle]at (8,2) (B') {$T^2A$};
		\node[nodestyle]at (8,12) (C') {$T^3A$};
		\node[nodestyle]at (-2,12) (D') {$T^4A$};
		\draw[thick,->](A')-- node[black, below]{\tiny $T^2e$}(B');
		\draw[thick,->](C')-- node[black, right]{\tiny $T\mu$}(B');
		\draw[thick,->](D')-- node[black, above]{\tiny $T^3e$}(C');
		\draw[thick,->](D')-- node[black, left]{\tiny $T\mu$}(A');
		\node[nodestyle]at (1,4) (a') {$T^2A$};
		\node[nodestyle]at (5,4) (b') {$TA$};
		\node[nodestyle]at (5,7.8) (c') {$T^2A$};
		\node[nodestyle]at (1,7.8) (d') {$T^3A$};
		\node[nodestyle]at (2.2,2.5) (a) {$TA$};
		\node[nodestyle]at (6,2.5) (b) {$A$};
		\node[nodestyle]at (6,6.5) (c) {$TA$};
		\node[nodestyle]at (2.2,6.5) (d) {$T^2A$};
		\draw[->](a')-- node[gray, above]{\tiny $Te$}(b');
		\draw[->](c')-- node[gray, left]{\tiny $\mu$}(b');
		\draw[->](d')-- node[gray, above]{\tiny $T^2e$}(c');
		\draw[->](d')-- node[gray, right]{\tiny $\mu$}(a');
		\draw[white,line width=7, -](d) -- (c);
		\draw[white,line width=7, -](d) -- (a);
		\draw[->](a)-- node[gray, above]{\tiny $e$}(b);
		\draw[->](c)-- node[gray, right]{\tiny $e$}(b);
		\draw[->](d)-- node[gray, above]{\tiny $Te$}(c);
		\draw[->](d)-- node[gray, left]{\tiny $\mu$}(a);
		\draw[->](A')-- node[gray, above]{\tiny $\mu$}(a');
		\draw[->](B')-- node[gray, above]{\tiny $\mu$}(b');
		\draw[->](C')-- node[gray, below]{\tiny $\mu$}(c');
		\draw[->](D')-- node[gray, below]{\tiny $\mu$}(d');
		\draw[white,line width=4, -](A) -- (a);
		\draw[white,line width=4, -](B) --(b);
		\draw[white,line width=4, -](C) -- (c);
		\draw[->](A)-- node[gray, below]{\tiny $\mu$}(a);
		\draw[->](B)-- node[gray, below]{\tiny $e$}(b);
		\draw[->](C)-- node[gray, above]{\tiny $\mu\ \ \ \ \ $}(c);
		\draw[->](D)-- node[gray, above]{\tiny $\mu$}(d);
		\draw[thick,->](A')-- node[black, below]{\tiny $T\mu$}(A);
		\draw[thick,->](B')-- node[black, above]{\tiny $Te$}(B);
		\draw[thick,->](C')-- node[black, above]{\tiny $\ \ \ T^2e$}(C);
		\draw[thick,->](D')-- node[black, above]{\tiny $\ \ T^2\mu$}(D);
		\draw[->](a')-- node[gray, below]{\tiny $\mu\ $}(a);
		\draw[->](b')-- node[gray, above]{\tiny $e\ $}(b);
		\draw[->](c')-- node[gray, above]{\tiny $\ \ Te$}(c);
		\draw[->](d')-- node[gray, below]{\tiny $T\mu\ \ $}(d);
		\draw[blue!20!white,line width=5, opacity=0.5](a) -- (b');
		\draw[blue,->](a)--(b');
		\draw[blue!20!white,line width=5, opacity=0.5](b') -- (c);
		\draw[blue,->](b')--(c);
		\draw[white,line width=7, -](c) -- (B);
		\draw[blue!20!white,line width=5, opacity=0.5](c) -- (B);
		\draw[blue,->](c) -- (B);
		\draw[white,line width=7, -](a) -- (c);
		\draw[green!20!white,line width=5, opacity=0.5](a) -- (c);
		\draw[blue,->](a)--(c);
		\draw[green!20!white,line width=5, opacity=0.5](b') -- (B);
		\draw[blue,->](b')--(B);
		\draw[white,line width=7, -](a) -- (B);
		\draw[green!20!white,line width=5, opacity=0.5](a) -- (B);
		\draw[blue,->](a)--(B);
		\draw[white,line width=7, -](D) -- (C);
		\draw[white,line width=7, -](D) -- (A);
		\draw[thick,->](A)-- node[black, below]{\tiny $Te$}(B);
		\draw[thick,->](C)-- node[black, right]{\tiny $Te$}(B);
		\draw[thick,->](D)-- node[black, above]{\tiny $T^2e$}(C);
		\draw[thick,->](D)-- node[black, left]{\tiny $T\mu$}(A);
	\end{tikzpicture}
\]
where the simplices of the bar construction can be conveniently visualized in terms of the schematic blue tetrahedron, extending \Cref{theta_cube} by one dimension.

\begin{thm}
	\label{no_horn_fillers}
	There is a finitary weakly cartesian monad $T$ on $\mathsf{Set}$ together with a $T$-algebra $A$ for which $\barc{T}{A}$ contains one of each kind of inner 3-horn without a filler.
\end{thm}

In fact, the examples that we construct in the proof will show that not only is it impossible to fill the interior of the 3-simplex, but even the remaining 2-simplex face cannot be filled.

\begin{proof}
	We again use the commutative-monoid monad $T$ and the bar construction of the $T$-algebra $A \coloneqq (\N,+)$. For the first kind of inner 3-horn as described above, consider
	\begin{align*}
		\alpha \;&\coloneqq\; 
		\boxed{ \boxed{ \boxed{2} + \boxed{2} } } 
		+ \boxed{ \boxed{ \boxed{2} } + \boxed{ \boxed{2} } }
		+ \boxed{ \boxed{ \boxed{3} } + \boxed{ \boxed{1} } }, \\[8pt]
		\gamma \;&\coloneqq\; 
		\boxed{ \boxed{ \boxed{2} + \boxed{2} } 
		+ \boxed{ \boxed{2} + \boxed{2} } }
		+ \boxed{ \boxed{ \boxed{3} + \boxed{1} } }, \\[8pt]
		\delta \;&\coloneqq\;
		\boxed{ \boxed{ \boxed{4} } } 
		+ \boxed{ \boxed{ \boxed{2} + \boxed{2} } 
		+ \boxed{ \boxed{3} + \boxed{1} } }.
	\end{align*}
	We verify that these 2-simplices indeed assemble to an inner 3-horn. We have that $\mu$ removes the outer boxes, $T\mu$ removes the mid-level boxes, and $T^2e$ removes the inner boxes (and possibly evaluates the sums). Hence indeed,
	\[
		(T\mu)(\alpha) \;=\; 
		\boxed{  \boxed{2}  + \boxed{2}  } 
		+ \boxed{ \boxed{2}  +  \boxed{2}  }
		+ \boxed{  \boxed{3}  +  \boxed{1}  } \;=\; \mu(\gamma),
	\]
	\smallskip
	\[
		(T^2 e)(\alpha) \;=\; 
		\boxed{ \boxed{4} } 
		+  \boxed{ \boxed{2} + \boxed{2} }
		+ \boxed{ \boxed{3} + \boxed{1} } 
		\;=\; \mu(\delta),
	\]
	and moreover,
	\begin{align}
	\begin{split}
		\label{gamma_delta}
		(T^2e)(\gamma) \;&=\; 
		\boxed{ \boxed{ 4 } 
		+  \boxed{ 4 }  }
		+ \boxed{ \boxed{ 4 } }     \\[8pt]
		(T^2e)(\delta) \;&=\; 
		\boxed{ \boxed{ 4 } } 
		+ \boxed{ \boxed{ 4 } 
		+ \boxed{ 4 } }
	\end{split}
	\end{align}
	resulting in $(T^2 e)(\gamma) = (T^2 e)(\delta)$, since these two expressions differ only by rearranging the summands.\footnote{This is where we are using that $T$ is not cartesian, but only weakly cartesian; the same example would not work with $T$ being the monoid monad, since the latter is cartesian and therefore its bar constructions are nerves of categories.}

	Hence $\alpha$, $\gamma$ and $\delta$ indeed define an inner 3-horn. Now if there existed a filler $\varepsilon\in T^4A$, then in particular there would have to be a $\beta \in T^3 A$ playing the role of the remaining 2-simplex, i.e.~satisfying the equations
	\[
		(T\mu)(\beta)=(T\mu)(\gamma),\qquad (T^2e)(\beta)=(T\mu)(\delta).
	\]
	Let's see why such a $\beta$ cannot exist. First of all,
	\begin{align*}
		(T\mu)(\gamma) \;&=\; 
		\boxed{  \boxed{2} + \boxed{2}  
		+  \boxed{2} + \boxed{2}  }
		+ \boxed{  \boxed{3} + \boxed{1}  }, \\[8pt]
		(T\mu)(\delta) \;&=\;
		\boxed{  \boxed{4}  } 
		+ \boxed{  \boxed{2} + \boxed{2}  
		+  \boxed{3} + \boxed{1}  }.
	\end{align*}
	Both terms consist of two outer boxes, and therefore $\beta$ must consist of two outer boxes as well. Thus, up to permutation,
	\[
		\begin{tikzcd}[row sep=large, column sep=-35]
			& \beta = \boxed{\vphantom{Jj}\cdots} + \boxed{\vphantom{Jj}\cdots} \ar[|->]{dl}[swap]{T\mu} \ar[|->]{dr}{T^2e} \\
			% (T\mu)(\gamma) = 
			\boxed{  \boxed{3} + \boxed{1}  } 
			+  \boxed{  \boxed{2} + \boxed{2}  
			+  \boxed{2} + \boxed{2}  }
			&&
			% (T\mu)(\delta) =
			\boxed{  \boxed{4}  } 
			+ \boxed{  \boxed{2} + \boxed{2}  
			+  \boxed{3} + \boxed{1}  }
		\end{tikzcd}
	\]
	Applying $Te$ to either desired result gives $\boxed{4}+\boxed{8}$. This shows that the two outer boxes of $\beta$ must match up with the other outer boxes as follows: in the first box of $\beta$, there must be an element of $TTA$ witnessing a partial evaluation from $\boxed{3} + \boxed{1}$ to $\boxed{4}$. In the second slot of $\beta$, we need to have a witness of a partial evaluation from $\boxed{2} + \boxed{2}  +  \boxed{2} + \boxed{2}$ to $\boxed{2} + \boxed{2} + \boxed{3} + \boxed{1}$. The proof is now complete upon noting that there is not such partial evaluation, for example because any nontrivial partial evaluation must strictly decrease the number of terms (\Cref{terms_decrease}).

	Concerning the second kind of inner 3-horn, consider similarly the terms
	\begin{align*}
		\alpha \;&\coloneqq\; 
		\boxed{ \boxed{ \boxed{2} + \boxed{2} } 
		+ \boxed{ \boxed{2} + \boxed{2} } } 
		+ \boxed{ \boxed{ \boxed{3} + \boxed{1} } }, \\[8pt]
		\beta \;&\coloneqq\; 
		\boxed{ \boxed{ \boxed{2} + \boxed{2} } } 
		+ \boxed{ \boxed{ \boxed{2} + \boxed{2} } 
		+ \boxed{ \boxed{3} + \boxed{1} } }, \\[8pt]
		\delta \;&\coloneqq\; 
		\boxed{ \boxed{ \boxed{4} + \boxed{4} } } 
		+ \boxed{ \boxed{ \boxed{4} } }.
	\end{align*}
	We have that
	\[
		\mu(\alpha) \;=\; 
		\boxed{ \boxed{2} + \boxed{2} } 
		+ \boxed{ \boxed{2} + \boxed{2} } 
		+ \boxed{ \boxed{3} + \boxed{1} }
		\;=\; \mu(\beta) 
	\]
	\[
		(T^2e)(\alpha) \;=\; 
		\boxed{ \boxed{4} + \boxed{4} } 
		+  \boxed{ \boxed{4} } 
		\;=\; \mu(\delta)
	\]
	and
	\begin{align*}
		(T^2e)(\beta) \;&=\; 
		\boxed{ \boxed{ 4 } } 
		+ \boxed{ \boxed{ 4 } 
		+ \boxed{ 4 } } \\
		(T\mu)(\delta)\;&=\;
		\boxed{  \boxed{4} + \boxed{4}  } 
		+ \boxed{  \boxed{4}  }
	\end{align*}
	which, as before, differ only by a permutation, and so are equal as elements of $TTA$. Therefore $\alpha$, $\beta$ and $\delta$ form an inner 3-horn. As before, we can show that we cannot even find a 2-simplex $\gamma\in T^3 A$ such that $\mu(\gamma)=(T\mu)(\alpha)$ and $(T\mu)(\gamma)=(T\mu)(\beta)$. Since
	\begin{align*}
		(T\mu)(\alpha) \;&=\; 
		\boxed{  \boxed{2} + \boxed{2} 
		+  \boxed{2} + \boxed{2}  } 
		+ \boxed{  \boxed{3} + \boxed{1}  } \\
		(T\mu)(\beta) \;&=\; 
		\boxed{  \boxed{2} + \boxed{2}  } 
		+ \boxed{  \boxed{2} + \boxed{2}  
		+  \boxed{3} + \boxed{1}  } 
	\end{align*}
	this would mean that we would have:
	\[
		\begin{tikzcd}[column sep=small,row sep=large]
			& \gamma \ar[|->]{dl}[swap]{\mu} \ar[|->]{dr}{T\mu} \\
			\boxed{  \boxed{2} + \boxed{2} 
			+  \boxed{2} + \boxed{2}  } 
			+ \boxed{  \boxed{3} + \boxed{1}  }
			&&
			\boxed{  \boxed{2} + \boxed{2}  } 
			+ \boxed{  \boxed{2} + \boxed{2}  
			+  \boxed{3} + \boxed{1}  } 
		\end{tikzcd}
	\]
	Hence upon considering $TA$ as a free $T$-algebra, the term on the right would have to be a partial evaluation of the term on the left. By the number-of-terms counting of \Cref{terms_decrease}, this is again not the case.
\end{proof}

\section{Compositional structure as completeness properties}
\label{seccompositional}

The results of the previous section have all been negative: we have only found counterexamples to the compositional structure of the bar construction in dimensions higher than 1.
In this final section, we focus on positive statements which hold for suitably well-behaved monads. The resulting compositional properties can be formulated in terms of filler conditions for simplicial sets, which we studied in detail in \cite{delta_squares}. We recall some of the main definitions and results from there, and show how to apply them to the concrete case of bar constructions of particular monads. 

We are mostly interested in when the commutative squares of structure maps in the bar construction form weak or strong pullbacks. %(From now on, ``square'' will always mean ``commutative square''). 
For convenience, we fix the following terminology; see the companion paper \cite{delta_squares} for more details.

\begin{defn}
 Let $X:\Delta^\op\to\Set$ be a simplicial set. We say that $X$ is:
 \begin{itemize}
	\item \emph{inner span complete} if it sends pushouts of coface maps in $\Delta$ to weak pullbacks in $\Set$ (\cite[Definition 5.2]{delta_squares}).
	\item \emph{stiff} if it sends those pushouts of one coface map and one codegeneracy map in $\Delta$ that are preserved by $\Delta \to \Set$ to pullbacks in $\Set$ (\cite[4.1]{GKT2}, \cite[Example 4.14]{delta_squares}).
	\item \emph{split} if it sends all pushouts of one coface map and one codegeneracy map in $\Delta$ to pullbacks in $\Set$ (\cite[5.1]{GKT2}, \cite[Example 4.15]{delta_squares}).
	%\item $X$ is \emph{span complete} if it sends the ``balanced'' squares of coface maps below to weak pullbacks in $\Set$
%  \item We say that $X$ is \emph{complete} with respect to a given collection of squares in $\Delta$ if the image in $\Set$ under $X$ of those squares is a weak pullback.
%  \item We call $\Comp(X)$ the set of squares of $\Delta$ sent by $X$ to weak pullbacks.
%  \item We say that $X$ is \emph{exact} with respect to a given collection of squares in $\Delta$ if the image in $\Set$ under $X$ of those squares is a (strong) pullback. 
%  \item We call $\Ex(X)$ the set of squares of $\Delta$ sent by $X$ to (strong) pullbacks.
 \end{itemize}
\end{defn}

To make these definitions more concrete, it helps to consider the following characterizations 
\cite[Theorem 4.11, Examples 4.14 and 4.15]{delta_squares}.

\begin{thm}\label{completeness_characterizations}
Let $X$ be a simplicial set.
\begin{itemize}
	\item $X$ is inner span complete if and only if the following squares are weak pullbacks:
		\[\begin{tikzcd}
			X_n \arrow{r}{d_i} \arrow{d}[swap]{d_j} & X_{n-1} \arrow{d}[swap]{d_{j-1}} \\
			X_{n-1} \arrow{r}{d_i} & X_{n-2}	\\
			{} \ar[phantom]{r}{(0 \le i < j - 1 \le n - 1)}		& {}
		\end{tikzcd}\]
	\item $X$ is stiff if and only if the following squares are pullbacks:
		\begin{equation}
			\label{stiff}
			\begin{tikzcd}
			X_n \arrow{r}{d_i} \arrow{d}[swap]{s_j} & X_{n-1} \arrow{d}[swap]{s_{j-1}} \\
			X_{n+1} \arrow{r}{d_i} & X_n	\\
			{} \ar[phantom]{r}{(0 \le i < j \le n)}		& {}
		\end{tikzcd}\qquad\begin{tikzcd}
			X_n \arrow{r}{d_i} \arrow{d}[swap]{s_j} & X_{n-1} \arrow{d}[swap]{s_j} \\
			X_{n+1} \arrow{r}{d_{i+1}} & X_n	\\
			{} \ar[phantom]{r}{(0 \le j < i \le n)}		& {}
		\end{tikzcd}\end{equation}
	\item $X$ is split if and only if it is stiff and the following squares are pullbacks:
		\begin{equation}
			\label{split}
			\begin{tikzcd}
			X_n \arrow[equals]{r} \arrow{d}[swap]{s_is_i} & X_n \arrow{d}[swap]{s_i} \\
			X_{n+2} \arrow{r}{d_{i+1}} & X_{n+1}	\\
			{} \ar[phantom]{r}{(0 \le i \le n)}		& {}
		\end{tikzcd}
		\end{equation}
\end{itemize}
\end{thm}

Thus, inner span completeness describes simplicial sets such that certain ``spans'' consisting of two simplices sharing a face map can always be filled into an $n$-simplex. Stiff simplicial sets are those for which any simplex with a degenerate \emph{spinal edge}, namely the edge between the $i$th and $(i+1)$st vertices, is itself degenerate in the manner that would produce such a degenerated edge. Split simplicial sets extend this condition to any degenerate edge, which adds on the condition of \emph{indecomposable units} (\cite[5.5]{GKT2}) in which any $n$-simplex whose composite edge from first to last vertex is degenerate is itself a degenerate $n$-simplex at a single vertex.

Applying \Cref{completeness_characterizations} with $n = 2$ shows that any two consecutive $1$-simplices in an inner span complete simplicial set can be filled to a $2$-simplex. Iterating this shows that any string of $1$-simplices spanning $n+1$ vertices can be ``composed'' into an $n$-simplex. While these properties are rather straightfoward, a more general and powerful result is that any \emph{directed acyclic configuration} with $n+1$ vertices can be completed to an entire $n$-simplex~\cite[Theorem 5.14]{delta_squares}.

Since, as we show below, algebras of BC monads have inner span complete bar constructions, it follows that every such bar construction has fillers for all directed acyclic configurations. This is our strongest result on the compositional structure of bar constructions.
We will also relate the additional properties of weakly cartesian and strictly positive monads to stiffness and splitness, respectively; splitness effectively describes the ``non-reversibility'' of partial evaluations in algebras of strictly positive monads.

\subsection{The bar construction for BC monads}

Let $(T,\eta,\mu)$ be a BC monad. Its definition weak pullback conditions translate directly to weak pullback properties of the standard commuting squares of generating structure maps in the bar constructions for algebras of $T$. For instance, the relations for face maps between the set of triangles $X_2 = T^3 A$, edges $X_1 = T^2 A$ and vertices $X_0 = TA$ are given by the following squares:
\begin{equation}
	\label{triangle_diagram}
	\begin{tikzcd} 
   		& & T A \\ 
   		& T^2 A \arrow{ur}{T e} \arrow{dl}[swap]{\mu} & T^3 A \arrow{l}[swap]{\mu} \arrow{r}{T^2 e} \arrow{d}[swap]{T \mu} & T^2 A \arrow{dr}{T e} \arrow{ul}[swap]{\mu} \\
  		T A & & T^2 A \arrow{rr}{T e} \arrow{ll}[swap]{\mu} & & TA
	\end{tikzcd}
\end{equation}
In order from bottom left to top to bottom right, these squares are instances for $\barc{T}{A}$ of the simplicial identities for face maps in dimension two, namely the following.
\begin{center}\begin{tikzcd}
	X_2 \arrow{r}{d_1} \arrow{d}[swap]{d_2} & X_1 \arrow{d}[swap]{d_1} \\
	X_1 \arrow{r}{d_1} & X_0
\end{tikzcd}\qquad\begin{tikzcd}
	X_2 \arrow{r}{d_0} \arrow{d}[swap]{d_2} & X_1 \arrow{d}[swap]{d_1} \\
	X_1 \arrow{r}{d_0} & X_0
\end{tikzcd}\qquad\begin{tikzcd}
	X_2 \arrow{r}{d_0} \arrow{d}[swap]{d_1} & X_1 \arrow{d}[swap]{d_0} \\
	X_1 \arrow{r}{d_0} & X_0
\end{tikzcd}\end{center}
If $\mu$ is weakly cartesian, then the top naturality square in~\Cref{triangle_diagram} is a weak pullback, hence so is the middle square above.  There is no reason to expect that the left or right squares would be weak pullbacks.

More generally, the square
\begin{center}\begin{tikzcd}
		T^{n+1} A \ar{r}{T^{n-i} \mu} \ar[swap]{d}{T^{n-j} \mu}		& T^n A \ar{d}{T^{n-j} \mu}	\\
		T^n A \ar{r}{T^{n-i-1} \mu}					& T^{n-1} A
\end{tikzcd}\end{center}
is $T^{n-j}$ applied to a naturality square of $\mu$ for any $0 \le i < j - 1 \le n - 1$, where for $i = 0$, it is understood that the horizontal maps are $T^n e$ and $T^{n-1} e$, respectively. The BC assumption implies that this square is weakly cartesian. Hence for $X = \barc{T}{A}$, the following square is a weak pullback when $i < j - 1$:
\begin{center}\begin{tikzcd}
	X_n \arrow{r}{d_i} \arrow{d}[swap]{d_j} & X_{n-1} \arrow{d}[swap]{d_{j-1}} \\
	X_{n-1} \arrow{r}{d_i} & X_{n-2}
\end{tikzcd}\end{center}
The analogous squares for $i=j-1$ are those given by associativity of $\mu$, multiplicativity of $e$, or some functor power $T^k$ applied to such a square, and do not need to be weak pullbacks in general.

In summary, the weak pullbacks among generating face maps in $\barc{T}{A}$ are then precisely those characterizing inner span complete simplicial sets in \Cref{completeness_characterizations}, which proves the following:

\begin{thm}
	\label{BC_to_ISpC}
	For any algebra $A$ of a BC monad $T$, the bar construction $\barc{T}{A}$ is inner span complete.
\end{thm}

Since inner span complete simplicial sets in particular have fillers for inner $2$-horns, this reproduces and generalizes the transitivity of the partial evaluation relation for BC monads from \Cref{pe_transitive}. More generally, it is worth reiterating that the existence of all fillers for directed acyclic configurations are implied, although for which we refer to~\cite[Theorem~5.14]{delta_squares} for the details.

\begin{ex}
 We saw in \Cref{distBC} that the distribution monad is BC. By \Cref{BC_to_ISpC}, this implies that for all its algebras, the bar construction is inner span complete. 
 At the lowest level, this in particular implies that the partial evaluation relation (known also as \emph{second-order stochastic dominance}, see \cite{FP}) is transitive.
 Inner span completeness is a stronger property than just inducing a transitive relation, and this may reflect a more profound structure at the level of random variables, generalizing the compositional nature of conditional expectation; see again~\cite{FP} for the relationship between partial evaluations and conditional expectation.

 However, a detailed analysis of the probabilistic meaning of inner span completeness is beyond the scope of this paper, and ideally would have to be carried out in categories other than $\cat{Set}$, facilitating the treatment of measure-theoretic probability.
\end{ex}

\subsection{The bar construction for weakly cartesian monads}

We now consider which additional compositional properties of the bar construction hold for weakly cartesian monads $(T,\eta,\mu)$ on $\Set$. In addition to the weak pullback squares which follow from the BC property for $T$, we therefore also have that $\eta$ is a weakly cartesian transformation. This yields in particular the weak pullback squares
\begin{equation}
	\label{Tstiff1}
	\begin{tikzcd}
		T^{n+1} A \arrow{r}{T^{n-i} \mu} \arrow{d}[swap]{T^{n-j+1} \eta} & T^{n} A \arrow{d}{T^{n-j+1} \eta} \\
		T^{n+2} A \arrow{r}{T^{n-i+1} \mu} & T^{n+1} A \\
		{} \ar[phantom]{r}{(0 \le i < j \leq n)}	& {}
	\end{tikzcd}
\end{equation}
where again $T^n \mu$ needs to be replaced by $T^n e$ for $i = 0$. Moreover, the squares
\begin{equation}
	\label{Tstiff2}
	\begin{tikzcd}
		T^{n+1} A \arrow{r}{T^{n-i} \mu} \arrow{d}[swap]{T^{n-j+1} \eta} & T^{n} A \arrow{d}{T^{n-j} \eta} \\
		T^{n+2} A \arrow{r}{T^{n-i} \mu} & T^{n+1} A	\\
		{} \ar[phantom]{r}{(0 \le j < i \leq n)}	& {}
	\end{tikzcd}
\end{equation}
which correspond to $T^{n-i}$ applied to a naturality square of $\mu$, are weak pullbacks already for any BC monad. 

In fact by \Cref{mono_strong_pullback}, as $T^k \eta$ is monic these squares are in fact strong pullbacks. These are precisely the squares in $\barc{T}{A}$ characterizing stiff simplicial sets~\eqref{stiff}, so by \Cref{completeness_characterizations} and \Cref{BC_to_ISpC} we have proved the following:

\begin{thm}\label{wc2ic}
For any algebra $A$ of a weakly cartesian monad $T$, the bar construction $\barc{T}{A}$ is both inner span complete and stiff.
\end{thm}

Since weakly cartesian monads include those arising from symmetric operads (\Cref{sym_operads_etc}), this result applies to many monads describing commonly occurring algebraic structures. It applies in particular to the commutative-monoid monad, for which we have studied the bar construction $\barc{T}{\N}$ in some detail in \Cref{secbar} and found counterexamples to various hypotheses about its compositional structure, such as the non-uniqueness of composites. We now have a positive result about its compositional structure, and many other bar constructions of a similar flavor, namely that of being a stiff and inner span complete simplicial set, with all the filler and degeneracy properties thus entailed. %Since this implies in particular inner span completeness, we can also apply \Cref{acyclicfiller} here, and conclude that fillers for acyclic configurations exist in $\barc{T}{\N}$, and more generally in all bar constructions of weakly cartesian monads.

\subsection{The bar construction for strictly positive monads}

Recall from \Cref{strictly_pos} that a monad $(T,\mu,\eta)$ is \emph{strictly positive} if for all $X$ the square
\[\begin{tikzcd}
	X \arrow[equals]{r} \arrow{d}[swap]{\eta\eta} & X \arrow{d}[swap]{\eta} \\
	T^2 X \arrow{r}{\mu} & T X
\end{tikzcd}\]
is a (necessarily strong by \Cref{mono_strong_pullback}) pullback.

If $X$ is the bar construction of a $T$-algebra $A$, then the square below left corresponds to the square below right.
\begin{center}\begin{tikzcd}
			X_{n} \arrow[equals]{r}{} \arrow{d}[swap]{s_{i+1} s_i}	& X_{n} \arrow{d}[swap]{s_i}		\\
			X_{n+2} \arrow{r}{d_{i+1}}				& X_{n+1}				\\
			{} \ar[phantom]{r}{(0 \le i \le n)}			& {}
\end{tikzcd}\qquad\quad\begin{tikzcd}
		T^{n+1} A \arrow[equals]{r}{} \arrow{d}[swap]{T^{n-i+1}(\eta\eta)} & T^{n+1} \arrow{d}{T^{n-i+1}\eta} A \\
		T^{n+3} A \arrow{r}{T^{n-i+1}\mu}				& T^{n+2} A				\\
		{} \ar[phantom]{r}{(0 \le i \le n)}			& {}
\end{tikzcd}\end{center}
These squares are all obtained by repeated application of $T$ to the strict positivity square above, and are therefore all pullbacks when $T$ is weakly cartesian and strictly positive. We then have the following by \Cref{completeness_characterizations} and \Cref{wc2ic}.

\begin{prop}\label{pwc2pc}
	For any algebra $A$ of a weakly cartesian and strictly positive monad $T$, the bar construction $\barc{T}{A}$ is both inner span complete and split.
\end{prop}

\begin{ex}
	For $M$ a monoid, the $M$-set monad $M \times -$ is cartesian. The bar construction of an $M$-set $A$ is therefore the nerve of a category, which we described in \Cref{monoid_action_bar}. If the unit of $M$ cannot be factored nontrivially, then the monad $M \times -$ is strictly positive (\Cref{monoid_spos}), and therefore $\barc{M\times -}{A}$ is split. This matches up with the fact that a nerve of a category is split if and only if no identity morphism can be factored nontrivially.
\end{ex}

By \Cref{nonpositive,semigroups}, all semigroups and commutative semigroups also have inner span complete and split bar constructions, when considered as algebras of the semigroup or commutative-semigroup monads respectively.

\begin{remark}
	Recall that the distribution monad is BC and strictly positive (\Cref{distBC,dist_monad_pos}), but not weakly cartesian (\Cref{distribution_monad_not_wc}). For monads of this kind, \Cref{pwc2pc} is ``almost'' true: the bar construction of any algebra is inner span complete, and is such that~\eqref{split} and the right square in~\eqref{stiff} are pullbacks for the same reasons as above. Thus what is missing for splitness (and stiffness) is only the class of squares as in the left of~\eqref{stiff}.
\end{remark}

\bibliographystyle{plain} % was: alpha
\bibliography{partial_evaluations}

\end{document}